\documentclass[a4paper,twoside,12pt]{article}      % Comments after  % are ignored
\usepackage{amsmath,amssymb,amsfonts} % Typical maths resource packages
\usepackage{amsthm}
\usepackage{graphics}                 % Packages to allow inclusion of graphics
\usepackage{color}                    % For creating coloured text and background
\usepackage{hyperref}                 % For creating hyperlinks in cross references
\usepackage{graphicx,eucal}
\usepackage{latexsym,epsfig,epic,eepic}

\newtheorem{theorem}{Theorem}[section]
\newtheorem{proposition}[theorem]{Proposition}

\newtheorem{lemma}[theorem]{Lemma}

\theoremstyle{definition}

\date{May 2010}

\title{Singular sets of holonomy maps for algebraic foliations}

\author{G. Calsamiglia, B. Deroin, S. Frankel \& A. Guillot}

\begin{document}
\maketitle
\begin{abstract}
In this article we investigate the natural domain of definition of a
holonomy map associated to a singular holomorphic foliation of
$\mathbb P^2$. We prove that germs of holonomy between algebraic
curves can have  large sets  of singularities for the analytic
continuation. In the Riccati context we provide examples with natural
boundary and maximal sets of singularities. In the generic case we
provide examples having at least a Cantor set of singularities and
even a nonempty open set of singularities.  The examples provided are
based on the presence of sufficiently rich contracting dynamics in the
holonomy pseudogroup of the foliation.
\end{abstract}

% ----------------------------------------------------------------

\section{Introduction}

%%%%%%%%%%%%%%%%%%%%

Let $\mathcal F$ be a singular holomorphic foliation of the complex projective plane. The leaves of the foliation may wind around  the plane in a very complicated manner. A standard approach to
understanding  such phenomena is to consider a real two dimensional subspace, $T$,
which intersects every leaf, and to study the equivalence relation on $T$ whose classes
are the intersections with each leaf. This  will be referred to as the \textit{holonomy relation}.

We classically take $T$ to be a compact (real) surface transverse to the foliation, and  in this case  one has  a refined construction whereby  leafwise paths with endpoints in $T$ determine, by lifting to nearby leaves, a pseudo-group acting on $T$, called the \textit{holonomy pseudo-group}, whose orbits are the holonomy classes. The study of this pseudo-group is central in foliation theory, insofar as it provides the dynamics  of the foliation, as well as some useful  geometric information.

As we shall see, a very natural choice for $T$ in our situation is to consider an algebraic curve, with finite tangency to $\mathcal F$. In fact   a generic curve $T$ intersects every leaf as follows from the maximum principle of~\cite{CLS}.  In this context,  holonomy classes are no longer the orbits of a pseudo-group of diffeomorphisms, due to the tangency of $T$ with the leaves. However, the situation is not so bad: the equivalence classes of the holonomy relation can still be described in terms of
% the orbits of
a set of local holomorphic correspondences   on $T$, the algebroid class of Painlev\'e, see~\cite{Painleve},\cite{Loray}.

There is a  special class of examples, which is  particularly  useful for testing any property of plane foliations, namely the Riccati differential equations, defined in affine coordinates $x,y$ by
 $dy = \omega _0 + \omega_1 y + \omega_ 2 y^2$, where the $\omega_i$ are rational $1$-forms in the $x$-variable. It is well-known that there is associated to this equation  a meromorphic flat connection on the fibration $(x,y) \mapsto x$, which   compactifies to a $\mathbb P ^1$-fibration (over a compact base space); the leaves of the foliation are, apart from some special vertical  fibers, the flat sections of the connection. Thus, if one chooses $T$ to be a vertical line, the holonomy classes are   orbits of a finitely generated group of automorphisms of $\mathbb P^1$, known as the monodromy group.

We cannot hope, for a general algebraic foliation of the plane, that the holonomy relation can be realized as a genuine group of automorphisms of some Riemann surface transverse to the foliation, even after allowing for
some (algebroid) ramified semi-conjugacy. Indeed, such a group necessarily preserves a projective structure on the surface (by uniformization), and it is well-known that the holonomy pseudo-groups of generic plane algebraic foliations are not this rigid; for instance,  given such a transversely projective structure  there must be a leaf that is contained in
%necessarily
an invariant algebraic curve, since $\mathbb P^2$ is simply connected, and this  is non-generic by Jouanolou's theorem \cite{Jouanolou} (see also~\cite{BLL}).

It
is possibly interesting,
 for the geometry  and dynamics of foliations, to consider weaker
analogues of
product covering structures,
%(e.g. completeness properties of holonomy; existence of complete lifts for leafwise paths to other leaves,  or of transversal paths, using holonomy, %to other transversals)
% for such foliations,
with respect to an algebraic fibration $\mathbb P^2 \rightarrow \mathbb P^1$, (or even a
smooth transverse plane field). In other words, to find a
%fibration such that the  pair foliation/fibration
lifting of paths in the leaves (resp. in the fibres) to paths in other leaves (resp. in other fibres)
which has good extension properties.  In this context it is natural to consider the problem of analytic  extensions of holonomy germs between fibres, the candidate germs for extension of holonomy being a priori analytic. This also naturally  justifies the  generalization from holonomy diffeomorphisms to   algebroid holonomy correspondences.   The extension problem may be expressed, in turn, in terms of
%determining
the maximal domain of  a holonomy correspondence defined  as an analytic multivalued map on a fibre $T$. In this vein, Loray conjectured in~\cite{Loray} (see also Ilyashenko's problem paper~\cite{Il}, especially Problem 8 and Problem 8.8 in~\cite{Il1}) that the analytic continuation of such a germ is possible along any path in $T$ that avoids a countable set of points, the singularities.

%It
%could, nevertheless, be very interesting
% for the geometry  and dynamics of foliations to find a weakened
%analogue of the
%product covering structures
%(e.g. completeness properties of holonomy; existence of complete lifts for leafwise paths to other leaves,  or of transversal paths, using holonomy, %to other transversals)
% for such foliations,
%with respect to some algebraic fibration $\mathbb P^2 \rightarrow \mathbb P^1$.  Note that in the context of analytic  extensions of holonomy it %is thus very natural to generalize from holonomy diffeomorphisms to the algebroid holonomy correspondences.   This problem may be expressed in %turn in terms of the maximal domain of definition of a holonomy correspondence. In this vein, Loray conjectured in~\cite{Loray} (see also %Ilyashenko's problem paper~\cite{Il}, especially Problem 8 and Problem 8.8 in~\cite{Il1}) the existence of analytic  extensions of
%holonomy, between such fibers, along
 %almost every path in $T$, and even every path which
%evitate
%avoids
%some countable set of singular points.

In his thesis~\cite{Marin}, Mar\'{\i}n found explicit examples where this conjecture can be verified.
In fact, his family can be extended to the family of all foliations defined by meromorphic closed $1$-forms on $\mathbb P^2$. These admit a multivalued holomorphic first integral whose monodromy preserves a Riemannian metric, and which allows to define, on each non-invariant algebraic curve $T$, a singular metric invariant by the holonomy pseudo-group. The problem of analytic continuation then reduces to
%the problem of
 the completeness of this metric.

The main result of our work is that,
on the other hand,
in the presence of rich contracting dynamics, which is known to be a generic property, there are special choices of the algebraic curve $T$ such that some holonomy correspondence has a maximal domain of definition with  large boundary.  We produce here examples whose boundaries contain Cantor sets of positive dimension, and even open sets. This disproves Loray's conjecture, but does not
%say anything about
exclude refining the
 problem to that of finding a
%good
better choice for $T$.

These
%bad
choices of $T$,  which would seem pathological in view of the aforementioned conjectures,
  reflect the fact that foliations with sufficienly rich holonomy have special transversals,
%which we
 called \textit{diagonals}, indicating perhaps
%duality
a hidden duality  involving
 the holonomy action on a   transversal, and
%another action on a leaf like object, e.g. base space of fibration for product cover.
the pseudo-group of transition functions along the leaves.
More precisely, we find some part of the foliation which can be described as a suspension, where the monodromy has a non-trivial domain of discontinuity when acting on the transversal space, with quotient   identifying with the base (and the diagonal).
So it is not yet clear whether the choices of diagonal $T$  should be regarded as  counterexamples
in a project aiming at finding fiber-like transversals, or rather viewed as an important aspect  of
the structure of $\mathcal F$.

 Basic examples of these non-fiber-like
%choices
transversal lines
arise noteably in the setting of
%foliations defined by a
Riccati equations,
where, as explained above,
%we know the existence of good ones.
we have  a suspension foliation with linear fibers,
and thus no lack of good transversal lines, realizing holonomy as a group of automorphisms.
%diffeos

%Here, it is possible to realize

\newcommand{\del}{\partial}

The
%simplest
most explicit
case
of a non-extendable holonomy map  that arises
is
the universal covering of the thrice punctured sphere,
%$\Sigma$,
$\mathbb H \subset \mathbb P^1 \rightarrow
\mathbb P^1\setminus \{0,1,\infty\} = \Sigma$,
realized as   holonomy   in
$S=(\widetilde{\Sigma} \times \mathbb  P^1)/\pi_1(\Sigma)$ ;
basically, at the level of the universal cover there is a
product foliation $\mathbb H\times \mathbb P^1$, and the
holonomy map goes from the disc $\{*\}\times \mathbb H$ in the
transversal fiber $\{*\}\times \mathbb P^1$ to the
(linear)
transversal diagonal in $\mathbb H\times \mathbb H$. This
illustrates the
%importance
non-trivial nature
of the choice of $T$,
(even  when constrained to be linear as in \cite{Loray})
and suggests the need for  further research.

Recall that this  suspension construction,  for
a circle bundle over a Riemann surface
 is also the most basic source of examples of
 Levi-flat hypersurfaces.  This is no coincidence;
the singular sets for holonomy extension in this article
 are
%typically
associated to
%the construction of a
laminations
within the holomorphic foliation,
%that produces
(albeit laminations   with  real 1-dimensional leaves,
modelled on the product of a Cantor set by a tree,
in the generic context of Theorem~\ref{thm:cantor}).
We recall
%also
that such suspensions and diagonals
have
also
been considered
%in the context
from the viewpoint of
 Stein spaces and harmonic sections
by~\cite{D-O}.
The
%embedding
birational maps that we realize of such suspensions into
$\mathbb P^2$ thus
provide   singular  Levi-flats with rich
dynamics in $\mathbb P^2$, as well as interesting examples of
 Stein spaces, and may provide further insight
into the general properties of Levi-flats  in $\mathbb P^2$.

%Diederich Ohsawa RIMS 1986

%\bibitem[D-O] {D-O} {\sc Diederich  \& Ohsawa.} stein. {\em RIMSh.} {\bf 117} (1986) no. , xxx-146.

%\bibitem [Palm] {Palm} {\sc Palmeira, .} titles. {\em Annals of  Mathematics} 389 (1980) p. 165-190.

Another interesting problem is to determine the difference
between the Riemann surface of a holonomy correspondence
obtained by analytic continuation, and the maximal domain where
the map is a holonomy map. Extending work of Mar\'{\i}n~\cite
{Marin}, we find particularly simple foliations defined by
closed $1$-forms, where for any holonomy germ these two domains
coincide, up to a countable number of points.
 However, it is conceivable that there exists a holonomy
correspondence whose domain of holonomy is considerably
smaller: consider $k$-fold
coverings
ramified along the diagonal,  $\Delta$ of
%again
$S_+=(\widetilde{\Sigma} \times \mathbb  H^2)/\pi_1(\Sigma)
\subset S$ ;   (taking $k=$  the Euler class of the circle bundle  $=c_1(N\Delta)$, up to sign, ensures existence)
%but now
%for example
 or of any such suspension type disc bundles with  holomorphic
diagonal,
%Stein domain  neighborhood Fuchsian
and extend a bit in $S$, beyond the Levi-flat boundary
$\del S_+$.
Then the leaf space of the cover over $S_+$ is just the disc;
i.e. the diagonal (in the cover), or the fiber
$\mathbb  H^2$ in $S_+$, but on the part of the cover beyond
$S_+$ the leaf space of the cover is
the $k$-fold cover of the annulus
%containing
outside of
$S^1=\del \mathbb   H^2 \subset \mathbb  P^1$.
Unfortunately we are unable at this time
 to provide such   examples in the algebraic context.
This non-Hausdorff leaf space structure is closely related to
considerations in the real codimension one case as studied
by Palmeira, see~\cite{Palm}, but it is not at all
clear what analogous structures one might find for
singular holomorphic foliations of the complex projective
plane.

While one might
%find an invitation to
   be tempted by a certain optimism
as regards the bridges connecting
analytic holonomy extension to
%these techniques of ~\cite{Palm, D-O},
techniques from the topology of real  codimension one foliations,
or the PDE theory of holomorphic extensions,
we should add that the subject seems to
be fraught with as many hidden traps as
  temptations, and should be
approached
with a certain degree of caution.

%-------------------------------

\subsection{Statement of results}

Let~$(C_0,p_0)$ and~$(C_1,p_1)$ be two pointed complex curves and let $h:(C_0,p_0)\to(C_1,p_1)$ be a germ of holomorphic function. A point~$q\in C_0$ is called a \textit{singularity of $h$} if there exists a path~$\tau:[0,1]\to C_0$ such that~$\tau(0)=p_0$, $\tau(1)=q$ and such that~$h$ admits an analytic
continuation along~$\tau([0,1))$ by germs of holomorphic maps from $C_0$ to $C_1$ but not along~$\tau([0,1])$ (see \cite{Loray}). The set of singularities for $h$ could, in principle, be any subset of $C_0$. If it is the whole of $C_0$ we will say that $h$ has \emph{full singular set}.

There may also exist an open set $D\subset C_0$ containing~$p_0$, such that for any path $\tau:[0,1]\to C_0$ starting at $p_0$ and such that~$\tau(1)\in\partial D$, $\tau^{-1}(D)=[0,1)$, the function $h$ has an analytic continuation along~$\tau([0,1))$ but does not have an analytic continuation along~$\tau([0,1])$. In the case where $\partial D$ is a topological circle, the analytic extension of $h$ is said to have $\partial D$ as its \textit{natural boundary}.

%Simple examples of singular points, not  in the natural boundary, can be obtained by considering
%a Riemann surface spread over a neighborhood $U$ of zero, with infinitely many sheets, such that the critical values
%accumulate only at zero, so there is clearly   a path
%going to infinity while its projection to $U$ converges to zero.

In this article, we investigate the set of singular points, of the germ of a holonomy map (or holonomy germ),  with domain
and range lying within algebraic curves, and that is associated to a given singular algebraic foliation $\mathcal F$ of the complex projective plane $\mathbb P^2$. We recall here the definition of holonomy germ. Let $C_0, C_1$ be algebraic curves of $\mathbb P^2$, and $p_0\in C_0$, $p_1\in C_1$ be points belonging to the same leaf $\mathcal{F}_{p_0}$ through $p_0$. We suppose that for $i=0,1$, $p_i$ does not belong to the singular set of $\mathcal F$, and that the curve $C_i$ is transverse to $\mathcal F$ at $p_i$. Consider a leafwise path $\gamma : [0,1]\rightarrow \mathcal{F}_{p_0}$  ,
such that $\gamma(0) = p_0$ and $\gamma (1) = p_1$. Then, one can find a continuous family of leafwise paths $\gamma^p$, parameterized by a point $p\in C_0$ close enough to $p_0$, such that $\gamma^p (0) = p$, $\gamma(1)$ belongs to $C_1$, and $\gamma^{p_0} = \gamma$. The germ of the map $p\mapsto h_{\gamma} (p)= \gamma^p (1)$ is uniquely determined by the relative (ie with fixed endpoints)  homotopy class of $\gamma$ under the above conditions, and is called the holonomy germ associated to $\gamma$. We will call such a germ between algebraic curves an \textit{admissible germ} for $\mathcal{F}$.

 \begin{theorem}
 \label{thm:countablesing}
   Let $\mathcal{F}$ be a singular foliation given by a closed meromorphic 1-form on $\mathbb P^2$. Then the set of singularities of an admissible germ for $\mathcal{F}$ between rational curves is at most countable.
 \end{theorem}

F.~Loray conjectured that the same is true for \emph{any} admissible germ associated to a holomorphic singular foliation of $\mathbb P^2 $~\cite{Loray}. The following result exhibits examples where the property fails:

\begin{theorem}
\label{thm:circle}
For any $d\in\{2,3,\ldots\}$, there exist holomorphic foliations of $\mathbb P^2$ in each of the following families:
\begin{enumerate}
\item \label{item:fuchsian} foliations  of degree $d$ having an admissible germ between lines with a natural boundary;
\item \label{item:dense} foliations  of degree $d+1$  having an admissible germ between lines with full singular set.
\end{enumerate}
\end{theorem}

Although  \ref{item:fuchsian}) seems to be a quantitatively weaker result than \ref{item:dense}), they are qualitatively different: the points in the natural boundary are singular for \textit{any} path leading to them whereas this is certainly not the case for all  singularities in  \ref{item:dense}).
Examples of item~\ref{item:fuchsian}) will be defined by a Riccati equation with a Fuchsian monodromy group in section~\ref{s:riccati}. The natural boundary appears as the limit set of the monodromy group, and it is   topological in nature: it is not possible to extend  the holonomy germ continuously to any of its points. Examples where the natural boundary is a fractal circle can also be constructed by using quasi-Fuchsian monodromy groups. The examples of item~\ref{item:dense}) are defined by using Riccati equations %with dense monodromy group.
with   monodromy group dense in Aut$(\mathbb P^1)$.

Note that the natural boundary here
%corresponds
is associated to
the Levi-flat boundary of a domain  in $\mathbb P^2$.
The constructions also provide examples of
topologically trivial, but analytically
non-trivial
deformations of foliations of $\mathbb P^2$ also admitting  Levi-flat hypersurfaces.
They are obtained by varying the moduli of
the set of punctures of the base space of the fibration, with respect to which the foliations are Riccati,
or by the   corresponding changes of moduli of the action on the fibre.
%Using simultaneous uniformization, we get quasi-Fuchsian monodromy groups,

%Along the proof it will be clear that the line of the range of the holonomy germis very special. However, by ``deforming'' it within the space of lines, we still get a very interesting phenomenon: there is an admissible germ still having a circle (resp. a nonempty open set) contained in its locus of singularities, but it will not be a natural boundary (resp. a full singularity set) anymore. This will be analyzed in detail in section~\ref{sec:bubbling} by using the bubbling techniques for projective structures and leads to

We point out that we do not know whether or not there exists a holomorphic foliation of the projective plane and an admissible germ from an algebraic curve to \textit{itself} which has a natural boundary. Our next result states that, however, part of the phenomena displayed by the foliations in the first item of Theorem~\ref{thm:circle} are well behaved under small perturbations of the transversal curves where the germs are defined, and in particular permits to construct such an admissible germ (i.e. from a line to itself) with a large singular set:

\begin{theorem}\label{thm:ricperturbed} There exists a foliation~$\mathcal{F}$ in~$\mathbb{P}^2$ and an open subset of pairs of lines~$P\subset(\mathbb{P}^2)^*\times (\mathbb{P}^2)^* $, intersecting the diagonal, such that for every~$(L_0,L_1)\in P$ there is an admissible germ from~$L_0$ to~$L_1$ whose analytic extension has a curve of singularities (in particular, an uncountable set).
\end{theorem}

The foliations in Theorems~\ref{thm:circle} and~\ref{thm:ricperturbed} are very special examples, for which we are even able to provide precise equations. Our next results concern the behaviour of admissible germs of most algebraic foliations of the plane. We prove that for a generic foliation, it is always possible to find an admissible germ from an algebraic curve to itself, whose set of singularities for the analytic continuation is a large set, i.e. of positive dimension, and in particular uncountable.

Recall that Jouanolou~\cite{Jouanolou} proved that there is an open and dense set in the quasi-projective manifold of degree $d$ foliations of $\mathbb P^2$ such that every foliation belonging to that set has no algebraic curve. It is rather classical to show that there is also a dense open set such that every foliation belonging to this set has only singularities with $\lambda_2/\lambda_1 $ non real, where $\lambda_1,\lambda_2$ are the eigenvalues (see~\cite{LinsNeto-Pereira}). We will prove that, for a foliation which belongs to these two sets, we can always find a holonomy germ from an algebraic curve to itself  whose set of singularities is large:

\begin{theorem} \label{thm:cantor}
Let $\mathcal{F}$ be a singular holomorphic foliation of $\mathbb P^2$ whose singularities are of hyperbolic type and with no invariant algebraic curve. Then there exists an admissible germ from a line to an algebraic curve whose singularity set for the analytic continuation contains a Cantor set.
\end{theorem}

It is also possible to construct an open set of foliations with  admissible germs whose singular set has nonempty interior, thus extending  the case of the foliation constructed in item \ref{item:dense} of Theorem \ref{thm:circle}. Loray and Rebelo exhibit for every integer $d\geq 2$, a non empty open set $\mathcal U_d$ in the parameter space of degree $d$ plane algebraic foliations, such that every foliation belonging to $\mathcal U_d$ has dense leaves, and other strong ergodic properties, see~\cite{LR}.

\begin{theorem}\label{thm:curveofsing} For every $d\geq 2$ and every singular holomorphic foliation belonging to a nonempty open subset of $\mathcal U_d$, there is an admissible germ from a line to an algebraic curve whose singularity set contains a nonempty open set.
\end{theorem}

We notice that the algebraic curve in Theorems~\ref{thm:cantor} and \ref{thm:curveofsing} can be made rational. A small perturbation of the union of the line with the original algebraic curve would provide a non singular rational curve, and after a birational change of coordinates, an example of an admissible germ from a line to itself with the same conclusions. We will however not provide the details here.

We finish by mentioning that Theorems \ref{thm:cantor} and \ref{thm:curveofsing} are also still true for small perturbations of the curves where the admissible germs are defined. The proof of this statement essentially follows the ideas used to prove Theorem~\ref{thm:ricperturbed} and will not be included in this work. Also, the same results most likely hold in the more general context of generic foliations on any complex algebraic surface, but we won't attempt to present things here at that level of generality.
\subsection{Organization of the paper}

In Section~\ref{sec:countablesing} we discuss the case of foliations defined by closed meromorphic $1$-forms, and prove Theorem~\ref{thm:countablesing}. The method is quite different from that of~\cite{Marin}, discussed in the Introduction. We prefer  here to use arguments involving Painlev\'e's Theorem I. We also prove that, in some special cases, the analytic extension of a holonomy germ is still a holonomy germ.  This section is essentially independent of what follows.

Theorems~\ref{thm:circle} and~\ref{thm:ricperturbed} are proved in Section~\ref{s:Riccati foliations}. Introductory material for the proofs is developed in Subsection~\ref{s:singsetforgerms}, where it is proved that the inverse of the developing map of certain projective structures on curves has large sets of singularities.
%and examples of germs with large sets of singularities are constructed.
%These are obtained as compositions of charts of branched projective structures on Riemann surfaces.
%These are obtained from developing maps of branched projective structures on Riemann surfaces.
%This material has been isolated here
%from the rest
It may be of independent interest as regards projective structures on Riemann Surfaces.
%The proof of Theorem~\ref{thm:circle} is given in Section~\ref{s:riccati}.
We also provide explicit polynomial forms for Theorems~\ref{thm:circle} and \ref{thm:ricperturbed} in Subsection \ref{sec:explicitformulae}.
%, and then prove Theorem~\ref{thm:ricperturbed}.

In Section~\ref{sec:cantor} we prove Theorems~\ref{thm:cantor} and~\ref{thm:curveofsing}. This Section can be read independently.

\subsection{Thanks}

We thank Yuli I'lyashenko, David Mar\'{\i}n and Paulo Sad for interesting discussions during the preparation of this work. The authors would also like to thank the institutions CRM (Bellaterra), UNAM (Cuernavaca), IMPA (Rio de Janeiro) for their hospitality. G.~Calsamiglia acknowledges support form FAPERJ/CNPq, B.~Deroin from the Ministerio de Educaci\'on y Ciencia de Espa\~na, ANR-08-JCJC-0130-01 (France),  ANR-09-BLAN-0116 (France) and from the  International Cooperation Agreement Brazil-France, A.~Guillot from CONACyT (Mexico) grant~58354 and PAPIIT-UNAM (Mexico) grant~IN102307. S. Frankel is grateful also to the Max Planck Institute for Mathematics in the Sciences, where part of this work was done.

%%%%%%%%%%%%%%%%%%%%%

\section{Foliations defined by closed 1-forms}
\label{sec:countablesing}

In this section we exhibit examples of foliations on projective manifolds for which admissible germs between rational curves posess small  singular sets  for  analytic continuation. Since the argument  only  uses the codimension, we obtain:

\begin{proposition}
\label{prop:countablesing}
  Let $\mathcal{F}$ be a codimension one holomorphic singular foliation defined by a meromorphic closed 1-form  $\omega$ on a projective algebraic manifold $M$ . Suppose that $L_0$ and $L_1$ are rational curves in $M$  and $h_{\gamma}:(L_0,p_0)\rightarrow (L_1,p_1)$ is an admissible germ for $\mathcal{F}$ associated to a leafwise path $\gamma$. Then the set of singularities of $h_{\gamma}$ is at most a countable set in $L_0$.

\end{proposition}
In particular, this is true for germs of holonomy associated to meromorphic fibrations $f:M\dashrightarrow\mathbb{C}P^1$.
Theorem~ \ref{thm:countablesing} is a direct consequence of Proposition \ref{prop:countablesing}.
 As we will see in Sections~\ref{s:riccati} and~\ref{sec:cantor} the hypothesis $d\omega=0$    is crucial. Notice that it allows one to define a transversely Euclidean structure for the foliation on a Zariski open subset. A question that this leaves open is whether Proposition~\ref{prop:countablesing} extends to foliations that admit a transverse Riemannian structure on an invariant (Zariski) open set.

%\begin{pfof}{Proposition \ref{prop:countablesing}}:

\begin{proof}
Let $\omega_i$ ($i=0,1$) be the meromorphic 1-form on the rational curve $L_i$ defined by restriction of $\omega$ to $L_i$. We claim  that the germs of 1-forms  $h_{\gamma}^*(\omega_1)$ and $\omega_0$ coincide at $p_0$. To prove the claim, consider $U$ the Zariski open set where $\omega$ is holomorphic. For any path $c$ in $U$ starting at $p_0=\gamma(0)$, the formula $F(c)=\int_{c}
\omega$ defines a multivalued holomorphic first integral of $\mathcal{F}_{|_{U}}$ on $U$ that forces the pole set $(\omega)_{\infty}$ to be invariant by $\mathcal{F}$. Hence, up to replacing $\gamma$ by its lift in a nearby leaf, we can suppose that $\gamma$ is contained in $U$ . Let $q\in L_0$ be a point sufficiently close to $p_0\in L_0$. Consider a foliated homotopy $H:[0,1]\times[0,1]\rightarrow U$ such that for all $s,t\in[0,1]$:
\begin{itemize}
  \item $H(s,0)=\gamma(s)$
  \item $\tau_0(t):=H(0,t)\in L_0$
  \item $\tau_0(1)=q$
  \item $H(s,t)$ belongs to the leaf of $\mathcal{F}$ through $\tau_0(t)$
  \item $\tau_{1}(t):=H(1,t)=h_{\gamma}(\tau_0(t))\in L_1$
\end{itemize}
 Since $\omega$ is closed on $U$ an application of Stokes' Theorem shows that the integral of $\omega$ on the path described by going once around the boundary of the square defined by $H$  is zero. On the other hand for any path $c$ in a leaf of $\mathcal{F}$, $c^*\omega\equiv 0$. Hence we have
$$\int_{\tau_0}\omega_0=\int_{\tau_0}\omega=\int_{\tau_1}\omega=\int_{h_{\gamma}\circ\tau_0}\omega_1=\int_{\tau_0}h_{\gamma}^*(\omega_1).$$

 The equality
 implies that $h_{\gamma}^*(\omega_1)=\omega_0$ as germs at $p_0$. Since meromorphic 1-forms on rational curves are rational, we can write the 1-forms in coordinates to find rational functions $R,S$ such that $\omega_0(x)=R(x)dx$ and $\omega_1(y)=S(y)dy$. The germ $y=h_{\gamma}(x)$ is a solution to the differential equation $$\frac{dy}{dx}=\frac{R(x)}{S(y)}.$$ A well known theorem of Painlev\' e (\cite{Painleve}) states that the \emph{solutions} of rational differential equations admit analytic extension along any path avoiding a countable set in the $x$ variable, i.e. in $L_0$. Hence the result.

\end{proof}

\subsection{Admissible analytic extension of an admissible germ; the examples of Mar\'{\i}n revisited}

In general it is not clear whether the analytic continuation of a holonomy germ is holonomic, and some care is needed
even in the case of foliations defined by global rational functions. For instance consider the foliation $\mathcal{F}=\{d(xy)=0\}$ and the lines $L_0=\{y=y_0\}$, $L_1=\{x=x_0\}$,  for fixed $x_0,y_0\neq 0$ close to $0$. The linear map $(x,y_0)\mapsto (x_0,x y_0/x_0)$ is a holonomy germ $(L_0,(x_1,y_0))\rightarrow L_1$ for $x_1\neq 0$ close to $0$ but not for $x_1=0$.

Nevertheless, in his thesis,  D. Mar\'{\i}n (\cite{Marin}) proved that for $\alpha\notin\mathbb{R}$ the foliation defined by $dy-y(y-1)(y-\alpha)dx=0$ in $\mathbb P^1\times\mathbb P^1$, which is Ricatti with respect to $dy=0$, satisfies that any holonomy germ $h$ from a non-invariant fibre $F$ of $dx=0$ to itself can be analytically extended along any path avoiding a countable set in the fibre (we could use Proposition~\ref{prop:countablesing} here) and, moreover, the extension is still a holonomy map.
%Eventually h
He further proves that any two such germs of holonomy, (different from the identity), can be obtained one from the other by analytic continuation along paths in the fibre. We are going to adapt this example to the present setting and generalize it.
\begin{proposition}
\label{prop:holonomicextension}
  Given a meromorphic 1-form $\eta$  on $\mathbb P^1$ consider the foliation $\mathcal{F}$ defined by the meromorphic closed 1-form $\omega$ on $\mathbb P^1\times\mathbb P^1$ whose expression in a chart is $$\omega(x,y)=dx-\eta(y)\quad\quad(x,y)\in\mathbb{C}^2$$ and an admissible germ $h$ for $\mathcal{F}$ between fibres of $\{dx=0\}$. If the analytic extension of $h$ along a path $\tau$ in the fibre that avoids $(\eta)_{\infty}$ is the germ of a  biholomorphism at $\tau(1)$, then it is also an admissible germ for $\mathcal{F}$.
\end{proposition}

%\begin{pfof}{Proposition \ref{prop:holonomicextension}}:

\begin{proof}
 Let  $\Sigma=\mathbb P^1\setminus\{p_0,\ldots,p_d\}$ be the set where $\eta$ is holomorphic.  Denote by $P:\widetilde{\Sigma}\rightarrow \Sigma$ a fixed universal covering map. By integration of $\eta$ along paths we can define a holomorphic mapping that corresponds to a branched projective structure (see definition in section \ref{s:singsetforgerms}) $\mathcal{D}:\widetilde{\Sigma}\rightarrow\mathbb C$  with critical points at the zero divisor $P^{-1}((\eta)_0)$ and with monodromy $\rho:\pi_1(\Sigma)\rightarrow\mathrm{PSL}(2,\mathbb C)$. Furthermore its monodromy group $\Gamma=\text{Im}(\rho)$ is a subgroup of translations of $\mathbb C$.

 Notice that $(\omega)_{\infty}$ is invariant by $\mathcal{F}$ and contains the set $\text{Sing}(\mathcal{F})$, so $\mathcal{F}$ is regular on  $\mathbb P^1\times\mathbb P^1\setminus(\omega)_{\infty}=\mathbb{C}\times \Sigma$. The map $(id,P):\mathbb{C}\times\widetilde{\Sigma}\rightarrow\mathbb{C}\times \Sigma$ is a universal covering map, and thanks to the closedness of $\omega$, the pull-back of $\mathcal{F}$ is a regular foliation $\widetilde{\mathcal{F}}$ on $\mathbb{C}\times \widetilde{\Sigma}$  defined by the holomorphic 1-form $$d(x-\mathcal{D}(y))=0\text{  where  }(x,y)\in\mathbb C\times \widetilde{\Sigma}.$$

  Let $\gamma$ be a leafwise path such that $h_{\gamma}=h$ and $\tau$ as in the statement of the proposition. Lift them to paths in $\mathbb{C}\times\widetilde{\Sigma}$ with common base point and, by abuse of language, reuse the names $\gamma$ and $\tau$. We have $\gamma(0)=(x_0,y_0)$ and $\gamma(1)=(x_1,y_1)$ and suppose that both are transversality points of $\widetilde{\mathcal{F}}$ with the respective fibres of $dx=0$.  The path $\tau$ is contained in the fibre $F_0=x_0\times \widetilde{\Sigma}$ and satisfies $\tau(0)=\gamma(0)$. Suppose $h_{\gamma}$ admits analytic extension along $\tau$ (this is the case for most paths thanks to Proposition \ref{prop:countablesing}).  We claim that the extended germs are still holonomy germs of $\widetilde{\mathcal{F}}$  at the points of $\tau$ where $F_0$ is transverse to $\widetilde{\mathcal{F}}$.

We have $\gamma(s)=(\gamma_1(s),\gamma_2(s))$ and $\tau(t)=(x_0,\tau_0(t))$. Let $c(t)=x_0-\mathcal{D}(\tau_0(t))\in\mathbb C$. The path $h_{\gamma}(\tau(t))=:(x_1,\tau_1(t))\in F_1$ satisfies $\tau_1(0)=\gamma_2(1)=y_1$ and $x_1-\mathcal{D}(\tau_1(t))=c(t)$ for small values of $t$, hence for all $t\in[0,1]$. Choose a homotopy $H:[0,1]\times [0,1]\rightarrow\widetilde{\Sigma}$ such that for each $(s,t)\in[0,1]^2$,
\begin{itemize}
 \item $H(0,t)=\tau_0(t)$
\item $H(1,t)=\tau_1(t)$
\item $H(s,0)=\gamma_2(s)$
\end{itemize}
Then the map $F(s,t)=(\mathcal{D}(H(s,t))+c(t),H(s,t))\in \mathbb C\times \widetilde{\Sigma}$ defines for each $t\in[0,1]$ a path $\gamma_t$ in a leaf of $\widetilde{\mathcal{F}}$ with an endpoint at $\tau(t)\in F_0$ and the other at $h_{\gamma}(\tau(t))\in F_1$. If neither of the endpoints belong to $\text{Tang}(\widetilde{\mathcal{F}},dx=0)$, the holonomy germ associated to $\gamma_t$ realizes the analytic extension of the holonomy germ $h_{\gamma}$  along $\tau$. Finally we need to project the homotopy back to $\mathcal{F}$.  It is enough to consider the composition $(id,P)\circ F:[0,1]^2\rightarrow\mathbb{C}\times\Sigma$.
\end{proof}
%\end{pfof}

%%%%%%%%%%%%%%%%%%%%%

\section{Riccati foliations on the plane}\label{s:Riccati foliations}

In this section we prove Theorems~\ref{thm:circle} and \ref{thm:ricperturbed}. We begin by some introductory material concerning projective structures on curves.

\subsection{Projective maps between surfaces}\label{s:singsetforgerms}

%\subsection{Examples of singular sets for germs between lines}
%\label{s:singsetforgerms}

In this section, we introduce the main material which will serve to prove Theorems~\ref{thm:circle} and~\ref{thm:ricperturbed}. We consider branched projective structures on Riemann surfaces $\Sigma_i$ for $i=0,1$, and a germ of projective map (i.e. a composition of charts of the branched projective structures) between $\Sigma_0$ and $\Sigma_1$. We aim to provide examples of such germs whose analytic continuation has a large set of singularities. We begin with some classical facts on projective structures.

Let $\Sigma$ be a Riemann surface. A \textit{branched projective
structure} on $\Sigma=\cup U_j$ is a collection $\{ \mathcal D_j
\}$ of non-constant holomorphic maps $\mathcal D_j : U _j \subset
\Sigma \rightarrow \mathbb P^1$  such that the change of
coordinates (where defined) are Moebius transformations, i.e.
$\mathcal D_j \circ \mathcal D_k ^{-1} (z)$ has the form $(az+b)/(cz+d)$. In
fact any of the functions $\mathcal D_j$ can be extended
analytically to a map $\mathcal D : \widetilde{\Sigma} \rightarrow
\mathbb P^1$ defined on the universal cover $\widetilde{\Sigma}$
of $\Sigma$. These \textit{developing maps} $\mathcal D$  are thus uniquely
defined up to post-composition by a Moebius map, and they satisfy a
formula of the form
\[  \mathcal D (\gamma (p) ) = \rho (\gamma) \mathcal D (p),\]
for every $p\in \widetilde{\Sigma} $ and every $\gamma \in \pi_1
(\Sigma)$, where $\rho: \pi_1(\Sigma) \rightarrow \mathrm{PSL}(2,\mathbb C)$
is a certain  representation associated to $\mathcal D$ called the
\emph{monodromy representation}. The monodromy representation, and hence its image, is well-defined up to
conjugacy by a Moebius map, and we will refer to the conjugacy class of
$\Gamma:=\rho(\pi_1(\Sigma))$ in $\mathrm{PSL}(2,\mathbb{C})$ as the
\textit{monodromy group} of the branched projective structure $\mathcal{D}$.
The term \emph{branched} comes from the fact that the maps $\mathcal{D}_i$ may
have critical points. In fact the $\mathcal{D}_j$ define an
orbifold structure on $\Sigma$ with projective change of
coordinates. If there are no critical points, they define a Riemann
surface atlas with projective change of coordinates and will be
simply called a \emph{projective structure}.

A projective structure on a punctured sphere $\Sigma=\mathbb{P}^1\setminus\{p_1,\ldots,p_n\}$ can be compared to the projective structure given by its natural embedding $\Sigma \hookrightarrow \mathbb P^1 $ in the Riemann sphere, via the \textit{Schwarzian derivative} defined for any function $f$ by $\{f,z\}= \frac{f'''}{f'} -\frac{3}{2}(\frac{f''}{f'})^2$. More precisely, to any projective structure $\mathcal{D}$ on $\Sigma$ is associated canonically a quadratic form $q(z) dz^2$ by the formula
\[ q(z)dz^2 : = \{ \mathcal D  , z \} dz^2 ,\]
which is independent of the chosen developing map. The quadratic form can also be defined in the case of a branched projective structure, but in this case it is meromorphic.

We will restrict our attention to parabolic type projective structures on $\Sigma$. Namely, in some coordinate $w$ around any puncture $p_i$, some developing map of our projective structure is given by
\[  \mathcal D  = \frac{1}{2\pi i}\log (w-w(p_i)) .\]
This property is equivalent to the following Laurent series  expansion
%term for
of the quadratic differential $q(z)dz^2$
\[ \{  \mathcal D, z \} dz^2 = \left( \frac{1}{2 (z-z(p_i) ) ^2} + \ldots \right) dz^2\]
at the neighborhood of any puncture $p_i$ (see~\cite{Hille}). The set $\mathcal P$ of parabolic type projective structures on $\Sigma$ is an affine complex space
%directed by
isomorphic to
the complex vector space of meromorphic quadratic differentials on the sphere, with poles only at the $p_i$'s of order at most $1$. This space has positive dimension as soon as the number of punctures is at least $4$.

Examples of parabolic type projective structures on punctured spheres are given by uniformizations. For our constructions, we will need projective structures of parabolic type on punctures spheres with monodromy groups dense  in $\mathrm{PSL}(2,\mathbb C)$:
\begin{lemma}
\label{lem:parabtypewithdensemonod}
If $\Sigma$ is the sphere punctured by at least four points, then there exists a parabolic type projective structure on $\Sigma$ with dense monodromy.
\end{lemma}

Before proving the lemma let us quickly sketch, for lack of a suitable reference, a proof that, in the case of punctured spheres, the monodromy of a parabolic type projective structure  is non elementary i.e. it is not conjugate to either, a subgroup of the affine group, a subgroup of $\mathrm{SU}(2)$ or to the group which preserves a non oriented geodesic in $\mathbb H^3$. Because  the  monodromy around the cusps is parabolic, it suffices to show that the monodromy is not conjugate to a subgroup of the affine group. But, if it were, then the foliation $\mathcal F$ constructed in subsection~\ref{s:riccati} would have a leaf that gives a section of the fibration $S\rightarrow \Sigma$ (the leaf through a fixed point of the monodromy in $\mathbb P^1$, the point at infinity for the affine action). Taking the compactification corresponding to the integers $n_p= 0$, (in subsection~\ref{s:riccati}) one shows by a computation that this leaf would compactify as an invariant algebraic curve passing through each saddle-node singularity of $\overline{\mathcal F}$. By the Camacho-Sad formula~\cite{Br}, the self-intersection of such an invariant curve vanishes. This would imply that the $\mathbb P^1$-bundle $\overline{S}$ would be biholomorphic to the product $\mathbb P^1 \times \mathbb P^1$, and would contradict the fact that there is a holomorphic section, $\overline{\Delta}$, of negative self-intersection. With this preliminary observation in hand, we proceed to prove Lemma~\ref{lem:parabtypewithdensemonod}.

\begin{proof}
For every element $\gamma \in \pi_1(\Sigma)$, consider the map $T_{\gamma}: \mathcal P \rightarrow \mathbb C$ which is defined as $T_{\gamma} (\sigma) = \mathrm{Tr}^2\  \rho_{\sigma}(\gamma)$, for any projective structure $\sigma$ whose monodromy is $\rho_{\sigma}$. This map is holomorphic. We claim that there exists a non trivial element $\gamma$ in $\pi_1(\Sigma)$ such that $T_{\gamma}$ is not constant, but if it were so, a well-known fact is that the representation $\rho_{\sigma}$ would be constant up to conjugation (see for instance \cite{HP}), equal to some representation $\rho$. Consider the $\mathbb P^1$-bundle constructed in Section \ref{s:riccati} using the representation $\rho$, the compactification being given by the choice of integers $n_p = 0$ for all cusps $p$ of $\Sigma$. Then the compactifications $\overline{\Delta_{\sigma}}$ of the diagonals defined by the projective structures $\sigma$ define a holomorphic family of different compact holomorphic curves, whose self-intersection is negative, by Equation~\ref{eq: self-intersection of diagonal} in Section \ref{s:riccati}: this is impossible since holomorphic curves intersect non negatively.  Thus, our claim is proven.

%Suppose the contrary. Then, a well-known fact (see~\cite{CS,HP}) shows that the monodromy representations are all equal to a same representation $\rho$, up to conjugation. This means that we have a $1$-parameter family of holomorphic sections in the foliated bundle constructed in section \ref{s:riccati} with compactification integers $n_{p_i}= 0$, whose self-intersection is negative. This is impossible.

Then, because $\mathcal P$ is an affine complex space, $T_{\gamma}$ must take one of the values $2\cos \alpha$ for an irrational real number $\alpha$, by Picard's Theorem. The monodromy $\rho_{\sigma}(\gamma)$ of the corresponding projective structure is conjugated to an irrational rotation, and thus the image of the monodromy representation is not discrete. Because the monodromy of a parabolic type projective structure is non elementary, this implies that up to conjugation, either the monodromy is dense in $\mathrm{PSL}(2,\mathbb R)$, or it is dense in $\mathrm{PSL}(2,{\mathbb C})$.

If it is dense in $\mathrm{PSL}(2,\mathbb R)$, then by perturbing the projective structure, the image of the monodromy will still be non discrete. Indeed, by Zassenhaus's lemma~\cite{Zas}, a non abelian subgroup generated by elements close enough to the identity is non discrete. Hence, by perturbing the projective structure so that $\mathrm{Tr}^2 (\gamma)$ is not real, the image of $\rho_{\sigma}$ will be dense in $\mathrm{PSL}(2,\mathbb C)$.
\end{proof}

We will fix a hyperbolic Riemann surface $\Sigma$ with a branched projective structure, a developing map $\mathcal D$, and study the analytic continuation of the inverse $\mathcal D^{-1}$. We will prove that the singular set for the analytic continuation of this map is the limit set of the monodromy group, when either the projective structure is the uniformizing one, or when the monodromy group is dense (in which case the limit set is the whole Riemann sphere).

\begin{proposition}
\label{prop:singofgerms}
  Let $\Sigma$ be a hyperbolic Riemann surface equipped with a branched projective structure and $\mathcal D$ be a developing map. We denote by $h$ a germ with extension the multivalued map $\mathcal D^{-1}$.
  \begin{enumerate}

    \item \label{item:singofgerms0} if the projective structure is that given by uniformization,then $h$ has a natural boundary for  analytic continuation.

    \item \label{item:singofgerms1} if the monodromy group is dense in $\mathrm{PSL}(2,\mathbb{C})$, then $h$ has full singular set.

  \end{enumerate}
\end{proposition}

Before proceeding to the proof let us remark that we are going to
find singularities of the inverse $h^{-1}$ based on an analysis of
the asymptotic values of the germ $h$. Given a holomorphic germ
$h:(\Sigma_0,p_0)\rightarrow(\Sigma_1,z_0)$ between complex curves,
we say $z\in\Sigma_1$ is an \emph{asymptotic value for $h$} if
there exists a path $\tau:[0,1)\rightarrow \Sigma_0$ such that
$\tau(0)=p_0$, the analytic continuation of $h$ along $\tau$ exists
and is a germ of biholomorphism at each point of $\tau$,
\[z=\lim_{t\rightarrow 1} h(\tau(t)) \text{ but }\lim_{t\rightarrow 1} \tau(t)\text{ does not exist}.\] Hence $h^{-1}$ does not admit continuous extension to $z$ along $h\circ\tau$, which is enough for $z$ to be a singularity of $h^{-1}$.
%\begin{pfof}{Proposition \ref{prop:singofgerms}}
 \begin{proof} Let us start with item (\ref{item:singofgerms0}). By the uniformization theorem for Riemann surfaces, $\Sigma$ is biholomorphic to the quotient of $\mathbb{D}$ by the action of a Fuchsian group preserving $\mathbb{D}$. The hypothesis states that, modulo conjugation by a Moebius map, $h^{-1}$ is a germ of the universal covering map $\mathbb{D}\rightarrow\Sigma$, hence it does not extend continuously to any point in the boundary of $\mathbb{D}$ .

To prove item (\ref{item:singofgerms1}) fix a point $p_0$ in
$\widetilde{\Sigma}$ over $0\in\Sigma$, and denote  $z_0= \mathcal
D (p_0)$. Let $\delta>0$ be a small real number, say $\delta =
1/10000$. On the sphere $\mathbb P^1 $, we consider the spherical
metric of constant positive curvature given in affine coordinates
$w$ by $|ds| = \frac{|dw|}{1+|w|^2}$. We denote by $Dh$ the
derivative of a Moebius map $h$ acting on the sphere, and by
$|Dh|(w)$ its spherical norm at the point $w$. If $w$ is a point of
the sphere, we denote by $w'$ the antipodal point.

For each point $z\in\mathbb P^1$ we will construct a finite set
$\mathcal A$ in $\pi_1(\Sigma)$, and an infinite sequence $\alpha_1,
\ldots, \alpha_n, \ldots$ of elements of $\mathcal A$, which has the
following properties. Denoting by $A_n := \alpha_1 \ldots \alpha_n$ (i.e. the group product), for every positive integer $n$, the diameter of the ball $$B_n
:= \{ w\in \mathbb P^1\ |\ |D \big( \rho(A_n) \big)|(w)\geq
1/2^{n}\}$$ tends to $0$ exponentially fast when $n$ tends to
infinity, and the image $\rho(A_n) (\mathbb P^1 \setminus B_n)$ is
contained in $B(z, cst/ 2^{n})$, where $cst$ is a universal
constant. Moreover, neither $z_0$ nor $\rho(\alpha_{n+1}) (z_0)$
belong to $B_n$ (observe that this implies in particular that
$\rho(A_n) z_0$ converges to $z$ when $n$ tends to infinity).

Before proving the existence of such a sequence $\alpha_n$, let us
explain why it implies the lemma. The idea is to consider, for every
non negative integer $n$, a smooth path $a_n : [0,1] \rightarrow
\widetilde{\Sigma}$, which begins at $p_0$ and ends at $\alpha_n
(p_0)$, of length bounded by a constant independent of $n$
(depending only on $\mathcal A$), and such that for a sufficiently
large integer $n$, $\mathcal D \circ a_n$ does not take any value in
the ball $B_{n-1}$. The reason why these paths $a_n$ exist is the
following. Consider in $\widetilde {\Sigma} $ a large ball $D$
containing all the points $\alpha (p_0)$, where $\alpha$ ranges over  $\mathcal A$. If $n$ is large enough, $B_{n-1}$ is a ball in
$\mathbb P^1$ with small diameter $|B_{n-1}|$, so that the set $D
\cap \mathcal D ^{-1} (B_{n-1})$ is a disjoint union of topological
balls of diameter bounded by $cst . |B_{n-1}|$, and their number is
bounded by the degree of $\mathcal D _{| D}$. To construct $a_n$, it
suffices to follow the geodesic $[p_0,\alpha_n(p_0)]$ between $p_0$
and $\alpha_{n} (p_0)$ in $\widetilde{\Sigma}$, and to
%surround
make a detour around the
components of $\mathcal D^{-1} (B_{n-1}) \cap D$ each time the
geodesic hits one of them.

We are ready to construct the path $c: [0,\infty) \rightarrow
\widetilde{\Sigma}$. This is an infinite concatenation of the paths
$c_n := A_{n-1} a_n$, which starts at $A_{n-1}(p_0)$ and ends at $A_n
(p_0)$; namely, we define $c (t) = c_n (t-n+1)$ for every non negative
integer $n$ and every $t\in [n,n+1]$. Notice that we have $\mathcal
D \circ c_n = \rho(A_{n-1}) \circ \mathcal D \circ a_n$, so that
because $\mathcal D \circ a_n$ does not take any value in the ball
$B_{n-1}$, the length of $\mathcal D \circ c_n$ goes to $0$
exponentially fast. We deduce that $\mathcal D \circ
c(t)$ converges to a point in $\mathbb P^1$ when $t$ tends to
infinity; this point must be $z$ since $c(n) = \rho(A_n) (z_0)$
converges to $z$ when $n$ tends to infinity, and thus $z$ is an
asymptotic value of $\mathcal{D}$.

It remains to prove the existence of the set $\mathcal A$ and of the
sequence $(\alpha_n)$.
We claim that it is possible to find a
finite subset $\mathcal B$ of $\pi_1 (\Sigma)$ such that, for every
pair of points $\{ u,v \}$ contained in $\mathbb P^1$ verifying
$d(u,v') \leq \delta$, there is an element $\beta \in \mathcal B $
such that
\begin{itemize}
\item $\rho (\beta)$ maps the points $u$ and $v$ to points $\delta/2$-close to $z_0$ and $z_0'$ respectively,
\item $| D \big( \rho(\beta) \big)| (u)\geq 4$ and $| D \big( \rho  (\beta) \big)| (v) \leq 1/4$.
\end{itemize}
Indeed, remark that for every pair $\{ u,v \}$ of points of the
sphere such that $d(u,v') \leq \delta$, there is a Moebius map $h$
in $\mathrm{PSL}(2,\mathbb C)$ such that $d(h(u), z_0) < \delta/2$, $d(h(v),
z_0')< \delta/2$ and $|Dh|(u) >4$, $|Dh|(v) < 1/4$. Because $\rho
(\pi_1(\Sigma))$ is dense in $\mathrm{PSL}(2,\mathbb C)$, we can suppose that
the Moebius map $h$ belongs to the image of $\rho$. But then, the
four preceeding conditions will be satisfied if we move $u$ and $v$
a little bit. Thus, the claim is a consequence of the compactness of
the set of pairs of points $\{u,v\}$ such that $d(u, v')\leq
\delta$.

We construct a sequence of elements $\beta_n$ in $\mathcal B$ in the
following way. The element $\beta_1$ is the element $\beta\in
\mathcal B$ corresponding to the choice $u= z$, $v= z'$. The element
$\beta_2$ is the element of $\mathcal B$ corresponding to the choice
$u= \beta_1 (z)$ and $v= \beta_1(z')$, etc. By construction,
$\beta_n \ldots \beta_1 $ maps $z$ and $z'$ to points
$\delta/2$-close to $z_0$ and $z_0'$ respectively, and moreover
\[  | D (\beta_n \ldots \beta_1 ) | (z ) \geq 4^n,\ \ \ \ \mathrm{and}\ \ \ \ | D (\beta_n \ldots \beta_1 ) | (z' ) \leq 4^{-n} .\]
Then we define $\mathcal A = \mathcal B ^{-1}$, $\alpha_n = \beta_n
^{-1}$. We will prove that this sequence verifies the required
conditions.

We need to analyze the behaviour of the function $|D \big( \rho(A_n)
\big) |$. Consider a Moebius map $R$ which sends respectively the
points $z$ and $z'$ to the points $\rho(A_n) ^{-1}(z) $ and
$\rho(A_n) ^{-1}(z')$, and which is $2$-bilipschitz. Such a map
certainly exists since the points $\rho(A_n)^{-1}(z)$ and
$\rho(A_n)^{-1}(z')$ are almost antipodal: more precisely, they are
$\delta/2$-close to $z_0$ and $z_0'$, so we have
$d(\rho(A_n)^{-1}(z) ', \rho(A_n)^{-1}(z') ) \leq \delta$. We will
study the derivative of the Moebius map $\rho(A_n) \circ R$ instead
of  $\rho(A_n)$. Observe that we have $| D ( \rho(A_n)
\circ R) | = |D R | . (|D \big( \rho(A_n) \big)| \circ R) $, and
thus $|D\big( \rho(A_n) \big)|$ is the same as $| D\big(
\rho(A_n)\circ R \big)| \circ R^{-1} $ up to a multiplicative
constant.  Changing affine coordinate if necessary without changing
the expression of the spherical metric, we can suppose that $z'=0$,
hence $z = \infty$. Because $\rho(A_n) \circ R$ fixes $z$ and $z'$,
it is a homothety: $\rho(A_n) \circ R (w) = \lambda . w$. A direct
computation shows that
\[ |D \big( \rho(A_n) \circ R \big) | (w) = \frac{1+|w|^2} {|\lambda| . |w|^2 + |\lambda |^{-1}}. \]
Observe that for every non negative integer $n$,
\[ |\lambda | = |D \big( \rho(A_n)  \circ R \big) | (0) \geq cst . |D \big( \rho(A_n) \big) | (\rho(A_n)^{-1} (z') ) \geq cst . 4^n.\]
Then, the derivative of $\rho(A_n) \circ R$ is larger than $cst .
2^{-n}$ if and only if $|w| \leq cst . 2^{-n/2}$. Hence, $B_n
\subset R (\{ |w|\leq cst . 2^{-n/2} \})$ and the diameter of $B_n$
tends exponentially fast to~$0$.

It remains to prove that neither $z_0$ nor $\rho(\alpha_{n+1})
(z_0)$ belong to $B_n$. We see that $|D\big( \rho(A_n) \circ R \big)
) |(w) \leq cst.|\lambda|^{-1}\leq cst. 4^{-n}$ if $w$ does not
belong to the ball of radius $\delta$ around $z'$. This implies that
$|D\big( \rho(A_n) \big)|$ is bounded by $cst. 4^{-n}$ outside the
ball $R (B(z', \delta)) \subset B (z_0' , 3\delta)$. Because $z_0$
does not belong to $B(z_0', 3)$, we have $|DA_n | (z_0 ) \leq cst .
4^{-n}$. Moreover,
$$|D ( \rho(A_n)  )| (\rho (\alpha_{n+1})
(z_0) ) = \frac{|D ( \rho(A_{n+1})  )|(z_0)}{ |D  (
\rho(\alpha_{n+1}) ) | (z_0) }\leq cst. 4^{-n}.$$
Hence, for $n$
sufficiently large, neither $z_0$ nor $\rho(\alpha_{n+1}) (z_0)$
belong to $B_n$. Item (\ref{item:singofgerms1}) is proved. \end{proof}

\subsection{Riccati foliation associated to a projective structure}\label{s:riccati}

Given any parabolic type projective structure on a punctured sphere $\Sigma$, we proceed to the construction of some associated  transversally projective foliations on rational complex surfaces. Fix a developing map $\mathcal D$ of the projective structure and the corresponding monodromy representation $\rho$. Consider the quotient of $\widetilde{\Sigma} \times \mathbb P^1$ by the action of $\pi_1 (\Sigma)$ on it given by $\mathrm{id} \times \rho$. There are three holomorphic structures invariant by this action:
\begin{itemize}
  \item The vertical fibration $\pi:\widetilde{\Sigma} \times \mathbb  P^1\rightarrow \widetilde{\Sigma}$,
  \item The horizontal foliation, transversal to the
  fibration $\pi$,
  \item The section
  $\Delta:=\{(x,\mathcal D(x)):x\in\widetilde{\Sigma}\}$, which is transverse to
  both $\pi$ and the foliation at every
%(non-branch)
point.
\end{itemize}
The quotient $S=(\widetilde{\Sigma} \times \mathbb  P^1)/\pi_1(\Sigma)$
is a non-compact surface equipped with a
$\mathbb P^1$-fibration $\pi:S\rightarrow\Sigma$, a
smooth holomorphic
foliation $\mathcal{F}$ transverse to $\pi$, and a holomorphic section $\Delta$
of $\pi$ transverse to $\mathcal{F}$.  For branched projective structures this has to be modified
by considering locally branched covers over the base, but this does not affect the compactifications that follows.

Our goal is to compactify~$(S,\pi,\mathcal F, \Delta)$, where $\pi$ compactifies as a rational fibration over $\mathbb P^1$, namely a Hirzebruch surface (see~\cite{Barth-al}), but with a singular foliation. Above a neighborhood of a cusp $p_i$, we have seen that the developing map $\mathcal D$ is given in some coordinates by $\mathcal D (w) = \frac{1}{2\pi i} \log w$, where $w$ induces a biholomorphism from a neighborhood $D_i$ of $p_i$ to the unit disc.  Thus, the fibration $\pi$  restricted to $\pi^{-1} (D_i)$ is defined as the quotient of $\mathbb H \times \mathbb P^1$ by the action $(x,z) \mapsto (x+1, z+1)$, where $x:= \frac{1}{2\pi i}\log w$, the foliation is the quotient of the horizontal foliation, and the curve $\Delta$ is the quotient of the diagonal $\{ z= w\}$.

There exist local models of $\mathbb P^1$-bundles over a disc with a meromorphic flat connection with a pole at $0$, and  parabolic monodromy (see \cite{Br}): for any non negative integer $n$, they are defined by the following differential equation
\begin{equation}\label{singularity}
w dt + (w^n - nt) dw =0,\end{equation}
in the coordinates $(w,t)$ belonging to ${\mathbb  D} \times \mathbb P^1$. These foliations have
either one or two singularities on the invariant
fiber $w=0$,
according to whether $n=0$ or $n>0$ respectively. In the $x$-coordinates, the equation (\ref{singularity}) reads $\frac{dt}{dx} = -e^{nx} + nt$, and
the solutions are given by $t(x)= (cst - x) \cdot  e^{nx}$. Thus, the map $(x,z) \mapsto ( x, t = (z-x) e^{nx} ) $ induces a biholomorphism between $\pi^{-1} (D_i)$ and $\mathbb D ^* \times \mathbb P^1$, sending the vertical fibration to itself, the foliation $\mathcal F$ to the foliation defined by (\ref{singularity}), and the curve $\Delta$ to the curve whose equation is $\{t = 0\}$.

Gluing the model (\ref{singularity}) to the surface $S$ for each cusp of $\Sigma$, we obtain a compact complex surface $\overline{S}$. It is equipped with a rational fibration $\mathbb  P^1 \rightarrow \overline{S} \stackrel {\overline{\pi}}{\rightarrow} \mathbb P^1$, and a singular holomorphic foliation $\overline{\mathcal F}$ transverse to $\overline{\pi}$, with $d+1$ invariant fibers defined by the models (\ref{singularity}). Moreover, the section $\Delta$ compactifies to a section $\overline{\Delta}$ of $\overline{\pi}$ which passes through each singularity of $\overline{\mathcal F}$ of type (\ref{singularity}) corresponding to a cusp $p$ with $n_p >0$, and
such that $\overline{\Delta}$
is transverse to $\overline{\mathcal F}$ at each point of intersection, of a fiber over a cusp $p$ of $\Sigma$, iff
$n_p=0$. We deduce   (see \cite{Br}) that the number of \lq\lq tangencies\rq\rq ,
\begin{equation}\label{eq:tangency}\mathrm{Tang} (\overline{\mathcal F}, \overline{\Delta} ) =\sum_p n_p.\end{equation}

We claim that the self-intersection of $\overline{\Delta}$ is
\begin{equation}\label{eq: self-intersection of diagonal} \overline{\Delta} ^2 = 2 + \sum_p (n_p -1).\end{equation}

To prove this fact, observe that the normal bundle of $\overline{\Delta}$ is canonically isomorphic to the tangent bundle of the fibration $\overline{\pi}$ restricted to $\overline{\Delta}$, because $\overline{\Delta}$ is everywhere transverse to $\overline{\pi}$. Consider the morphism
$P: T \Delta \rightarrow T  F|_{\Delta}$, from the tangent bundle of $\Delta \subset \overline{\Delta}$ to the tangent bundle of the fibration $F$,
defined as the projection along  the tangent bundle of $\mathcal F$. In the coordinates $(w,t)$ above a neighborhood of the cusp $p$, the projection $P$ has the following form:
%\[  P\left(\frac{\partial}{\partial w}\right) - w^{n_p -1} \frac{\partial}{\partial t}.\]

\[  P\left(\frac{\partial}{\partial w}\right) = w^{n_p -1} \frac{\partial}{\partial t}.\]

This means that $P$ extends   meromorphically from $T\overline{\Delta}$ to $TF|_{\overline{\Delta}}\simeq N_{\overline{\Delta}}$, with a $0$ of order $n_p-1$ at the cusp $p$, and we get the desired formula.

We recall here that by deforming the conformal structures of the punctured sphere, associated to base or fibre, we obtain foliations
that are holomorphically different and remark that only a very special subset of these admit
holomorphic diagonals~\cite{D-O}. This provides some motivation to consider smooth real two-dimensional transversals,
with the conformal structure they inherit from the foliation.

\subsection{Proof of Theorem \ref{thm:circle}}
Choose $d\in\{2, 3, \ldots\}$. Consider the Riemann sphere punctured at $d+1$ points, together with its parabolic type \emph{uniformizing} projective structure: take a curvilinear
$(d+1)$-sided polygon in the unit disc $\mathbb{D}\subset\mathbb P^1$ with vertices in the
unit circle $\mathbb{S}^1$ and   consecutive sides $C_i$, $i=0,\ldots,d$, that are subarcs of circular arcs
perpendicular to $\mathbb{S}^1$ at the vertices.
Let $r_i$ be the hyperbolic reflection with respect to the circle
containing $C_i$ and let $G$
denote the (Fuchsian) group of Moebius transformations whose generators are the compositions
$\rho_i=r_i\circ r_{i+1}$ for $i=0,\ldots,d$ where $r_{d+1}:=r_0$.
By Poincar\'e's theorem (see for instance~\cite{Beardon}), $G$ leaves $\mathbb{D}$ invariant and acts freely
discontinuously there. The quotient map $\mathbb{D}\rightarrow
\mathbb{D}/G=:\Sigma$ is a universal covering map of the curve $\Sigma$, which is biholomorphic
to a $(d+1)$-punctured sphere, and we consider the projective structure given by the
canonical equivariant embedding
%of its universal covering space in the Riemann sphere
$\mathcal{D}:\widetilde{\Sigma} = \mathbb{D}\hookrightarrow \mathbb P^1$.
By construction $\mathcal{D}$ is a parabolic type projective structure and by Proposition \ref{prop:singofgerms} a germ of its inverse $\mathcal{D}^{-1}$ has the unit circle $\mathbb{S}^1$ as natural boundary.

For the proof of item (\ref{item:dense}) consider a parabolic type projective structure $\mathcal{D}$ on a $d+2$ punctured sphere $\Sigma$ with dense monodromy, which exists by Lemma~\ref{lem:parabtypewithdensemonod}.  Proposition~\ref{prop:singofgerms} guarantees that a germ of the inverse $\mathcal{D}^{-1}$ has full singular set.

Construct the holomorphic foliation $\overline{\mathcal{F}}$ associated to the pair $(\Sigma,\mathcal{D})$, using the method
at the beginning
of this section,
%\ref{s:riccati}
by choosing
%all but one of the $n_p$ to be  $1$, and the other one to be zero.
one $p_0$, setting  $n_{p_0}=0$ and for all other  $p$, setting  $n_p=1$ .
Hence the compactification $\overline{\Delta}$ of the diagonal section $\Delta$ associated to $\mathcal{D}$ is  a smooth rational curve of self intersection $+1$. This implies that the rational surface is the first Hirzebruch surface $\mathbb{F}_1$, which has a unique
%$(-1)$
exceptional-curve disjoint from $\overline{\Delta}$. Blowing down gives  $\mathbb P^2$ equipped with a holomorphic foliation $\mathcal{G}$ of degree $d$ in the case of a Fuchsian representation, and of degree $d+1$ in the case of dense representation (by (\ref{eq:tangency})). By construction there is an admissible germ for $\mathcal{\mathcal{G}}$ from  a fibre of the rational fibration (a line in $\mathbb P^2$) to the image of $\overline{\Delta}$ (another line in $\mathbb{P}^2$) that corresponds to the germ $\mathcal{D}^{-1}$ of $\overline{\mathcal{F}}$ and having  the corresponding set of singularities as desired, because the restriction of the blow down to a fibre (resp. $\overline{\Delta}$) is  biholomorphic onto its image.
%\end{pfof}
%\end{proof}

By using a different  set of techniques  we provide, in section \ref{sec:explicitformulae}, explicit polynomial 1-forms for the  foliations constructed here.

\subsection{Explicit expressions}\label{sec:explicitformulae}

In some particular cases, we can obtain birational models for the previously constructed Riccati foliations. We claim that we can find all the previous phenomena within the family of degree four foliations of~$\mathbb{P}^2$ given in an affine chart by the kernel of the form
\begin{equation}\label{eq:ric-original}(y^2[2\lambda x+1]-2x^3+L)ydx+(x^4-2xy^2[\lambda x-1]+y^4-Lx)dy,\end{equation}
for $L = (1+\lambda^2)y^2x^2-2\lambda y^4$ and~$\lambda\in\mathbb{C}$. Blowing up the origin we pass from~$\mathbb{P}^2$ to the first Hirzebruch surface, which is a rational fibration with rational base.
In the coordinate~$x=ty$ of the blowup we obtain the foliation~$\mathcal{F}$ given by
\begin{equation}\label{eq:ric-explicit}\left[1+2t(\lambda -t^2)y+(t^2+t^2\lambda^2-2\lambda)y^2\right]dt-(t^4-1)dy.\end{equation}

In order to understand the foliation we will consider the foliation~$\mathcal{F}_2$ that~(\ref{eq:ric-explicit}) induces is the \emph{second} Hirzebruch surface~$\mathbb{F}_2$, by means of the chart~$(t,y)=(T^{-1},YT^2)$. The rational fibration~$\Pi:\mathbb{F}_2\to\mathbb{P}^1$ determined by~$\Pi(y,t)=[t:1]$ is transverse to~$\mathcal{F}_2$ except at the four exceptional fibers~$\{t^4-1\}$: the foliation is a Riccati one. In other words, if we let~$\Sigma=\mathbb{P}^1\setminus\{1,-1,i,-i\}$ and~$\Omega=\Pi^{-1}(\Sigma)$, $\Pi|_\Omega:\Omega\to\Sigma$ is a rational fibration everywhere transverse to~$\mathcal{F}_2$. There is a section~$\Delta$ of the fibration~$t\mapsto(0,t)$ and this section can be compactified to~$\overline{\Delta}=\{y=0\}$, which is everywhere transverse to the foliation. Since we are in a transversely projective foliation, $\Delta$ inherits a projective structure that can compared to the canonical one:

\begin{proposition}\label{prop:schwexplicit} The Schwarzian derivative of the projective structure induced by~$\mathcal{F}$ with respect to the canonical one in~$\Delta$ is given by
\begin{equation}\label{eq:schw-explicit}2\frac{4t^2+\lambda(t^4-1)}{(t^4-1)^2}dt^2\end{equation}
\end{proposition}

\begin{proof}
Recall that if $H(t)=\left(\begin{array}{cc} a(t) & b(t) \\ c(t) & d(t) \end{array}\right)$ is a solution to the linear differential equation
\begin{equation}\label{relations}\frac{d}{dt}H(t)=\left(\begin{array}{rr}\frac{1}{2}\alpha_1 & \alpha_0 \\ -\alpha_2 & -\frac{1}{2}\alpha_1 \end{array}\right)H(t),\end{equation}
the general solution to the Riccati equation
\begin{equation}\label{ric}[\alpha_2(t)y^2+\alpha_1(t)y+\alpha_0(t)]\frac{\partial}{\partial y}+\frac{\partial}{\partial t},\end{equation}
is given by
$$h(t)(y_0)=\frac{a(t)y_0+b(t)}{c(t)y_0+d(t)}.$$
This implies that the flow of the vector field~(\ref{ric})
in time~$s$ is
$(y,t)\stackrel{s}{\longrightarrow}(h(t+s)h^{-1}(t)(y),t+s)$ and hence
$(y,t)\stackrel{t_0-t}{\longrightarrow}(h(t_0)h^{-1}(t)(0),t_0)$.
The holonomy from~$\Delta$ to the fiber~$t=t_0$ is in consequence, up to a Moebius transformation, given by
$\phi(t)=h^{-1}(t)(0)=-b(t)/a(t)$. We will now calculate its Schwarzian derivative. By~(\ref{relations}), we have that
$$\frac{\phi''}{\phi'}=\frac{\alpha_0'}{\alpha_0}-\alpha_1-2\frac{c}{a}\alpha_0.$$
Developing~$\{\phi(s),s\}=(\phi''/\phi')'-\frac{1}{2}(\phi''/\phi')^2$
using~(\ref{relations}) and the above formula, we obtain
$$\{\phi(t),t\}dt^2=\left[2\alpha_0\alpha_2-\frac{1}{2}\alpha_1^2-\alpha_1\frac{\alpha_0'}{\alpha_0}-\alpha_1'-\frac{3}{2}\left(\frac{\alpha_0'}{\alpha_0}\right)^2+\frac{\alpha_0''}{\alpha_0}\right]dt^2.$$
It is well-defined even in the presence of multivaluedness for the  flow of the vector field.
For equation~(\ref{eq:ric-explicit}), this is exaclty~(\ref{eq:schw-explicit}).
\end{proof}

Notice that for every fourth root of unity~$t_0$, the expression~(\ref{eq:schw-explicit})  is, at~$\tau=t-t_0$, of the form
$(\frac{1}{2}\tau^{-2}+\cdots)d\tau^2$ and this implies that, at each one of these four points in~$\overline{\Delta}\setminus\Delta$, the monodromy is parabolic (the point at~$t=\infty$ is a regular point of the projective structure).

The monodromy representation~$\rho:\pi_1(\Delta)\to \mathrm{PSL}(2,\mathbb{C})$ of the projective structure \emph{is} the group of the Riccati equation. If~$\gamma:[0,1]\to\Delta$ is a closed path, since the holonomy from~$\Delta$ to a generic fiber~$F$ is a developing map~$\mathcal{D}$ of the projective structure and~$\mathcal{D}(\gamma\cdot p)=\rho(\gamma)\mathcal{D}(p)$, the difference of the holonomies is exaclty~$\rho(\gamma)$.

We can pull-back the fibration~$\mathbb{P}^1\to\Omega\to \Sigma$ to~$\widetilde{\Sigma}$ via~$\Pi:\widetilde{\Sigma}\to\Sigma$ and obtain in this way a fibration~$\mathbb{P}^1\to\widetilde{\Omega}\to \widetilde{\Sigma}$ endowed with a foliation~$\widetilde{\mathcal{F}}_2$ that is transverse to the fibration. Up to a biholomorphism we can suppose that~$\widetilde{\Omega}=\widetilde{\Sigma}\times\mathbb{P}^1$, and that~$\widetilde{\mathcal{F}}_2$ is the horizontal foliation in~$\widetilde{\Sigma}\times\mathbb{P}^1$. We may recover~$\Omega$ by acting diagonally by~$\pi_1(\Sigma)$ upon~$\widetilde{\Sigma}\times\mathbb{P}^1$ by deck transformations in the first factor and by some representation~$\mu:\pi_1(\Sigma)\to \mathrm{PSL}(2,\mathbb{C})$ in the second one. This last representation must be the monodromy one. Furthermore, in~$\widetilde{\Sigma}\times\mathbb{P}^1$ the section~$\Delta$ becomes the graph of a function~$f:\widetilde{\Sigma}\to\mathbb{P}^1$, but the projection onto the second coordinate must be the developing map. We conclude that the foliations considered here belong to the family of previously constructed foliations. They correspond to the choice of~$n_p=0$ in each one of the four punctures of~$\Sigma$. The surface~$\mathbb{F}_2$ appears as a consequence of~formula~(\ref{eq: self-intersection of diagonal}). \\

Since the dimension of the space of parabolic type projective structures on a four punctured sphere is one, equation~(\ref{eq:schw-explicit})  is a parametrization of that space in the case~$\Sigma=\mathbb{P}^1\setminus\{1,-1,i,-i\}$. For the projective structure corresponding to the (Fuchsian) uniformization, the quadratic differential must be invariant by the group of biholomorphisms of~$\Sigma$. Under the symmetry~$t\mapsto it$ the quadratic differential~(\ref{eq:schw-explicit}) is the same except for a change of sign in~$\lambda$, this is, only $\lambda=0$ may give a quadratic differential invariant by the biholomorphisms of~$\Sigma$ and corresponds thus to the uniformization parameter. Among the rest of parameters we find quasi-Fuchsian representations (for sufficiently small $\lambda$) and dense representations, as was shown in Lemma \ref{lem:parabtypewithdensemonod}.

\subsection{Proof of Theorem~\ref{thm:ricperturbed}}

We will consider the foliation~(\ref{eq:ric-original}) with~$\lambda=0$ and the lines~$C_{p,\epsilon}$ given by~$\{x-yp-\epsilon=0\}$ for~$|p|<\frac{1}{4}$. After the blowup leading to the foliation~(\ref{eq:ric-explicit}), the strict transform of~$C_{p,\epsilon}$ becomes the curve~$y=\epsilon(t-p)^{-1}$. This curve does not intersect~$\Delta$. Let~$K\subset\overline{\Delta}$ be the annulus~$\frac{1}{2}<|t|<2$. The complement of~$K$ in~$\overline{\Delta}$ is a compact set free of singularities of the projective structure (the four punctures are contained in~$K$). If~$\epsilon$ is small enough, $C_{p,\epsilon}$ is close to~$\overline{\Delta}$ along~$K$ and thus, since~$\overline{\Delta}$ is transverse to~$\mathcal{F}$,  we have a holonomy map~$f:K\to C_{p,\epsilon}$ that is a biholomorphism onto its image. Let~$q\in\mathbb{C}$, $q^4\neq 1$. We have shown that there exists an open disk~$D\subset \{t=q\}$ and a holonomy map~$h:D\to \Delta$ that realizes the Fuchsian uniformization of~$\Delta$ and that has~$\partial D$ as natural boundary. Notice that, for every~$z\in\partial D$ we can find a path~$\gamma:[0,1]\to \{t=q\}$ such that~$\gamma(1)=z$ and such that~$\gamma(t)\in h^{-1}(K)$ for~$t<1$.  Consider the restriction~$h|_{h^{-1}(K)}:h^{-1}(K)\to K$. It is still a holonomy map and, through~$\gamma$, has in~$z$ a singularity for its analytic continuation. Hence, for every~$z\in\partial D$, $f\circ h|_{h^{-1}(K)}$ is a holonomy map from~$\{t=q\}$ to~$C_{p,\epsilon}$ for which~$z$ is a singularity for its analytic continuation (see Figure~\ref{fig:schw-perturb}). Notice that, because~$\mathcal{F}$ is both transverse to~$\Delta$ and to~$\{t=q\}$ at the intersection of these curves, $\Delta$ intersects~$\{t=q\}$ in the interior of~$D$. Let~$B\subset D$ be a small closed ball containing this point of intersection. If~$\epsilon'$ is small enough, $\mathcal{F}$ establishes a holonomy diffeomorphism $g:(\{t=q\}\setminus B)\to C_{q,\epsilon'}$ defined, in particular, in~$\partial D$. We have the holonomy~$f\circ h|_{h^{-1}(K)}\circ g^{-1}$ between~$C_{q,\epsilon'}$ and~$C_{p,\epsilon}$. The holonomy is defined along~$g\circ\gamma(t)$ for every~$t<1$ but cannot be extended to~$t=1$. In other words, \emph{the points of the curve $g(\partial D)$ are singularities for the analytic continuation of an admissible germ from~$C_{q,\epsilon'}$ to~$C_{p,\epsilon}$}. This ends the proof of the Theorem. \\

It is worth noticing that, if~$\epsilon$ is small enough, the projective structure in~$C_{p,\epsilon}$ can be completely understood: It is a bubbling (see~\cite{GKM}) over the projective structure of~$\Delta$ along two arcs, one of them close to~$\{t=p\}$ and the other close to~$\{t=\infty\}$. This is, however, unnecesary for the proof of the Theorem.

%%%%%%%%%%%%%%%%%%%%%%%%%%%%%%%%

\begin{figure}
\centering
\includegraphics[width=5in]{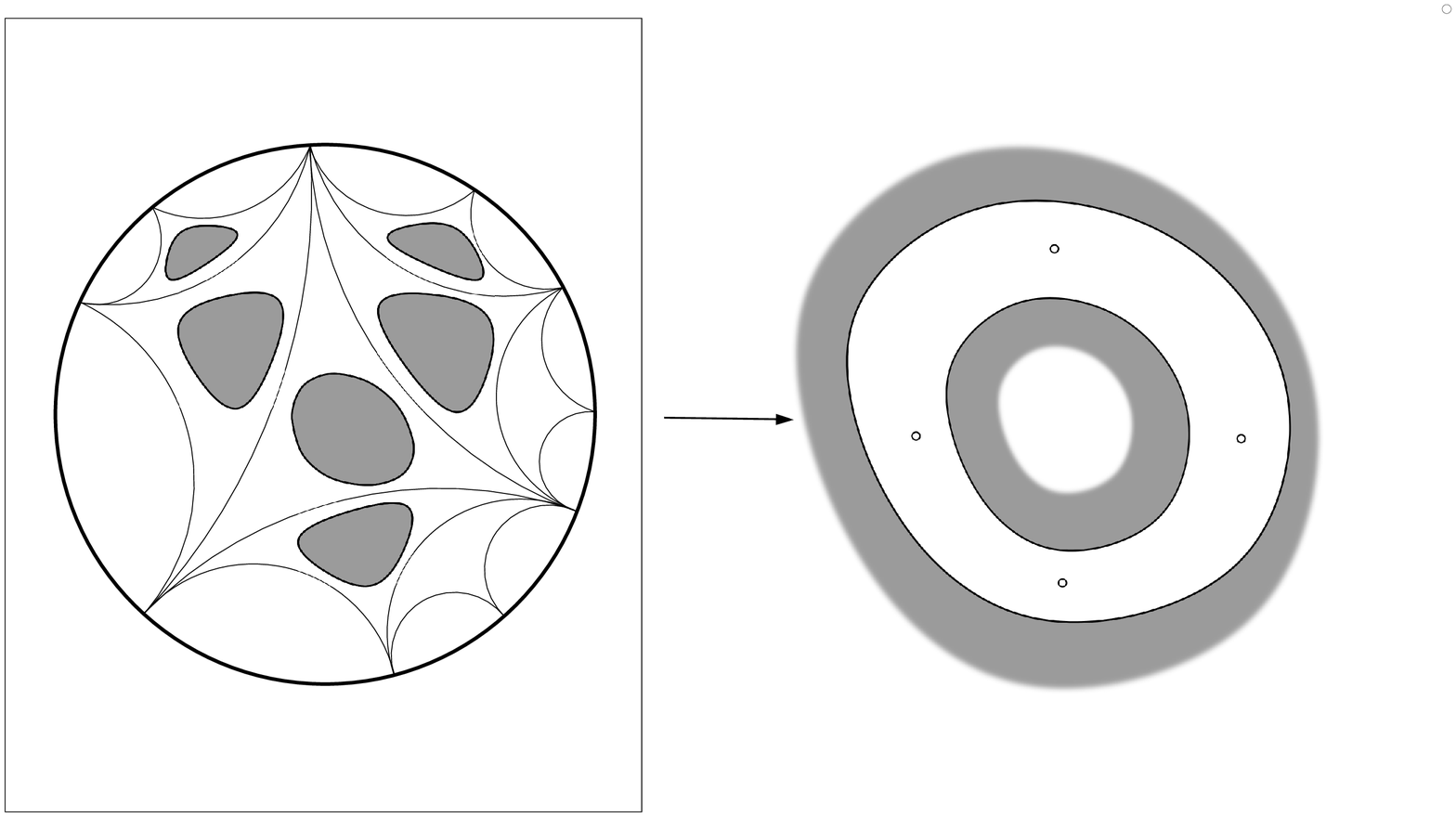}
  \caption{Singularities of the holonomy in the perturbed line: through the holonomy relation, the complement of the shaded part within the circle realizes the universal covering of the annulus. This covering cannot be extended to the circle.}\label{fig:schw-perturb}
\end{figure}

\section{Generic foliations}\label{sec:cantor}
\subsection{Proof of Theorem~\ref{thm:cantor}}

The holonomy map with  natural boundary (Theorem \ref{thm:circle}) was
%achieved by realizing
constructed using
the quotient map of the domain of discontinuity of a Fuchsian group. Analogous maps, closely related to Schottky groups, can in fact be seen as the holonomy of generic foliations of the complex projective plane. The
%first tool that will facilitate their construction is the
existence of rich dynamics in the holonomy pseudo-group will underlie their construction.

Before proceeding to an explanation, we need to introduce the concept of $\mathcal F$-homotopy, and to discuss a subtlety in the definition of  the holonomy map. Given a compact topological space $X$, an algebraic foliation of $\mathbb P^2$, and continuous maps $F_0,F_1: X \rightarrow  \mathbb P^2$, an $\mathcal F$\textit{-homotopy} between $F_0$ and $F_1$ is a continuous family of leafwise paths $\Gamma^x: [0,1] \rightarrow \mathbb P^2 $ such that $\Gamma^x (i) = F_i (x)$ for $i=0,1$ and every $x\in X$. Sometimes the $\mathcal F$-homotopy will be denoted by $\Gamma(s,x) = \Gamma^x(s)$. Now, define a \textit{holonomy tube} between two open sets $U,V$ contained in  algebraic curves, by an $\mathcal F$-homotopy between $id_{\bar U}$ and some map $h:{\bar U} \rightarrow {\bar V}$. Of course, the map $h$ associated to a holonomy tube is a holonomy map, but strictly speaking tubes have more structure, the reason being that in a holonomy map the homotopy class of the leafwise path is not uniquely determined.

\begin{lemma}\label{lem:schottky}
Under the hypotheses of Theorem \ref{thm:cantor} there exists a line $L$, an open disc $B\subset L$ and a pair of holonomy tubes between $\overline{B}$ and the pair of subsets,  $h_i \overline{B}\subset B$, where the  associated holonomy maps $h_1,h_2:\overline{B}\rightarrow B$ are injective and such that
$$h_1(\overline{B})\cap h_2(\overline{B})=\emptyset.$$
\end{lemma}
\begin{proof}
  Let $\mathcal{M}$ be a pseudo-minimal set for  $\mathcal{F}$, that
  is, a nonempty closed invariant subset of
  $\mathbb  P^2\setminus{\emph{Sing}(\mathcal{F})}$ such that every
  leaf
    $\mathcal{F}_x \subset\mathcal{F}$ is dense in
  $\mathcal{M}$. Existence  of $\mathcal{M}$ is standard. It can be shown by using
 the
fact that no leaf of $\mathcal{F}$ is contained in a ball,
so no leaf is localized near a singularity (see \cite{Zak}).
  We will now produce
a loop
  $\gamma:{\mathbb {R}/\mathbb {Z}}\rightarrow \mathcal{F}_x\rightarrow\mathcal{M}$,
for some $x$,
and a transverse open disc $B$ contained in a line in $\mathbb  P^2$
such
that the holonomy map
$h_{\gamma}$ satisfies
$h_{\gamma}(\overline {B}) \Subset  B$.

Suppose that $\overline{\mathcal{M}}$ contains
 a singularity $p$ of $\mathcal{F}$. By
hypothesis $p$ is hyperbolic, so
restricting  to   the domain of linearization of $\mathcal{F}$ at
 $p$, a  leaf   of
$\mathcal{M}$
accumulates in a  separatrix $S$  of $p$,
 so
 $S\subset\mathcal{M}$.
 There is  then a loop $\gamma\subset S$   and a transverse disc
$B$ producing the hyperbolic holonomy germ
$\overline {h_{\gamma}(B)} \Subset  B$ as desired.
On the other hand, if $\overline{\mathcal{M}}$ does not
contain a singularity of $\mathcal{F}$, then $\mathcal{M}$ is an
exceptional minimal set and the existence of the germ $h_\gamma$ was
proven in \cite{BLM}.
In any case ${h_{\gamma}}$ has a fixed point $0$,
(in a coordinate $z:B\rightarrow\mathbb{D}$)
and by  the Schwarz lemma
${h^n_{\gamma}(B)} \rightarrow 0$ uniformly in $n$. Shrinking $B$ if necessary, we can construct a holonomy tube from $\overline{B}$ to $h_{\gamma}(\overline{B})$, that we denote by $\Gamma':[0,1]\times B \rightarrow\mathbb P^2$.

Since $\mathcal{F}$ has no invariant algebraic curves, $\mathcal{M}$
is transversally a perfect set. We can choose open sets (by increasing $n$ as needed)
$U,V\subset B$  such that
$z=0\in U\Subset\overline{U}\Subset h(B)$,
$\overline{V}\Subset B\setminus \overline{h(B)} $,
and there exists a holonomy
diffeomorphism $\psi:U\rightarrow V$ with $\psi(0) \in V $. Shrinking $U$ if necessary, we consider a holonomy tube $\Gamma'': [0,1]\times U \rightarrow \mathbb P^2$ from $U$ to $V$ whose associated map is $\psi$.
 Now
for sufficiently large $n\in\mathbb{N}$,
$h_{\gamma}^{n}(B)\Subset U\cap  \psi^{-1}(V) $,
so  by
defining
$h_1=h_\gamma  $
and $h_2=\psi\circ h_{\gamma}^{n}  $,
we get two injective holonomy maps $h_1,h_2:\overline{B}\rightarrow B$ of $\mathcal{F}$ satisfying $$h_1(\overline{B})\cap h_2(\overline{B})=\emptyset.$$
We define the tubes $\overline{\Gamma^1} = \Gamma'$ and  $\overline{\Gamma^2} = (\Gamma')^{*n} * \Gamma''$, whose associated holonomy maps are respectively $h_1$ and $h_2$.
\end{proof}

We denote by $H = \langle h_1,h_2\rangle $ the semi-group generated by the two holonomy maps of Lemma~\ref{lem:schottky}. The action of $H$ on $B$ is a local semi-group version of the action of a Schottky group on $\mathbb{P}^1$. In particular
%it is well-known that
there exists a unique closed $H$-invariant set, and   it is  contained in the closure of every $H$-orbit. This     \textit{limit set}   is
\[  \Lambda_H := \bigcap_{n\in \mathbb N^*} \left( \bigcup_{(i_1,\ldots i_n) \in \{1,2\}^n} h_{i_1}\circ \ldots \circ h_{i_n} (\overline{B}) \right) .\]
The following  fairly standard Lemma is  included  for lack of a suitable reference:
\begin{lemma}\label{lem:limit}
  Under the hypothesis of~\ref{lem:schottky}, the set $\Lambda_H$ is a Cantor set in $B$. Moreover, $H$ acts freely discontinuously on $U_H=B\setminus \Lambda_H$.
\end{lemma}
\begin{proof}
 Let $N$ be a natural number and for each finite ordered set $I=(i_n)_{n\le N}\in\{1,2\}^{N}$ define $h_I=h_{i_1}\circ\cdots\circ h_{i_{N}}\in H$, and say it has {\it length} $N$. By a trivial induction, each $h_I:B\rightarrow B$ is an injective holomorphic mapping. We need to analyze the set of points where orbits of the action accumulate, i.e., $\Lambda_H=\bigcap_{k} \bigcup_{N(I)=k}  h_I(B)$. For each $I$ the set $ B\setminus h_I(B)$ is conformally equivalent to a unique annulus
$$A_{I}=\{z\in {\bf C}: r_I<|z|<1\},$$
of modulus
$$m(B\setminus h_I(B))=-\log(r_I).$$

Set $0<c<\min \{m(B\setminus h_1(B)), m(B\setminus h_2(B)) \}$. We claim that $m(B\setminus h_I(B))>c\cdot N$. Indeed, this can be proved by induction on the length. For $N=1$ it is obvious. Suppose $I'$ has length $N-1$ and satisfies  $m(B\setminus h_{I'}(B))>c\cdot(N-1)$. Let $I=I'i_{N}$. Then by using superadditivity of  moduli (see \cite{ahl}) we have that
$$m(B\setminus h_I(B))\ge
m(B\setminus h_{I'}(B)) + m(h_{I'}(B)\setminus h_I(B)).$$
On the other hand
 $$m(h_{I'}(B)\setminus h_I(B)) > m(B\setminus h_{i_{N}}(B))>c$$
 and therefore,  $m(B\setminus h_I(B))>c\cdot N$.

Hence for an infinite sequence, $I=(i_n)_{n\in\mathbb{N}}\in\{1,2\}^{\mathbb{N}}$, if we consider its truncation $I_{N}$ of length $N$ we have that $\lim_{N\rightarrow\infty}m(B\setminus{h_{I_{N}}(B))}=\infty$ and there exists a unique point $k_I\in \bigcap_{n\in\mathbb{N}}h_{i_1}\circ\cdots\circ h_{i_n}(B)\in B$. It is then clear that the limit set of $H$,
$$\Lambda_H=\bigcup_{I\in\{1,2\}^{\mathbb{N}}}k_I$$
is a totally disconnected, perfect set in $B$, hence a Cantor set.
\end{proof}

%%%%%%%%%%%%%%%%

We now proceed to the proof of Theorem \ref{thm:cantor}. Suppose that we
could find an algebraic curve $C$ and a holonomy map $h : B\setminus \Lambda_H \rightarrow C$ which is invariant by the action of the semi-group $H$. Then, the map $h$ would not extend continuously to any point of $\Lambda_H$, and Theorem~\ref{thm:cantor} would follow.
While this seems to be asking too much,
a slight weakening of this property, by restricting $h$ to an $H$-invariant tree in $ B\setminus \Lambda_H$,
is sufficient for our purposes.  The tree structure is already implicit in the multi-indices used
to parametrize points of  $\Lambda_H$ in the proof of Lemma~\ref{lem:limit}.
We proceed to  constructions of such  $h$ based on  suspensions  using  semigroup actions, and
 containing trees as
diagonals.  We hope this also provides some understanding of the  link to our results
 in the previous sections.

%In general, finding such a couple $(C,h)$ might be impossible. Nevertheless a

Consider the bouquet of two oriented circles $Q= S_1 \vee S_2$, with base point $*$, and let $\pi_1 (Q)^+ \subset \pi_1(Q)$ be the positive semi-group generated by the classes $[S_i]$, $i=1,2$ in the fundamental group. Denote by $T$ the Cayley graph of $\pi_1(Q)^+$ associated to the set of generators $[S_1], [S_2]$; this is a dyadic rooted tree and $\pi_1(Q)^+$ acts on $T$ by the restriction of the $\pi_1(Q)$ action on its Cayley graph. We will denote by $\rho : \pi_1(Q) ^+ \rightarrow H$ the representation defined by sending $[S_i]$ to $h_i$, which clearly gives a semigroup isomorphism.  The next proposition provides in particular the tree in $ B\setminus \Lambda_H$, on which our suspension analogy will be built.
\begin{proposition} \label{p:F-homotopy}
There exist two continuous maps $ F_i : T \rightarrow \mathbb P^2$, $i=1,2$ such that
\begin{enumerate}
\item $F_0$ takes values in $ U_H= B\setminus \Lambda_H$ and is $\rho$--equivariant,
\item \label{it:algcur} $F_1$ takes values in a smooth algebraic curve $C$ of $\mathbb P^2$ and is  $\pi_1(Q)^+$--invariant, thus giving an  immersed image of $Q$,
\item $F_0$ and $F_1$ are $\mathcal F$-homotopic.
\end{enumerate}
\end{proposition}
Assuming this proposition, let us prove Theorem \ref{thm:cantor}. Denote by $h: U\subset B\rightarrow V\subset C$ a germ of holonomy map defined by the $\mathcal F$-homotopy between $F_0$ and $F_1$. Remark that by equivariance each point $q\in \Lambda_H$ can be approached by an infinite path in $F_0 (T)$ that, apart from continuity, has very little regularity at $q$ (it is essentially log--fractal, i.e. fractal wrt the $log(r)$ transformation of polar to cylindrical coordinates). Nevertheless, along such a path, the holonomy $h$ will not have any limit, since it is a curve winding infinitely many times around the two circles of the bouquet $F_1(T)$.  (Recall the discussion of asymptotic values following proposition~\ref{prop:singofgerms}.) Thus $\Lambda_H$ is a subset of the set of singularities for the topological continuation of $h$.

It remains to prove Proposition~\ref{p:F-homotopy}. We will first produce a pair $F_0,F_1$ satisfying all the desired properties except that $F_1$ does not take values in an algebraic curve.  A basic approximation method will allow us to correct this easily.

 To guarantee the continuity of some of the constructions we will repeatedly use that, by work of Lins-Neto in~\cite{Lins}, for the class of foliations in Theorem~\ref{thm:cantor}, the leaves of $\mathcal{F}$ are of hyperbolic type and the unique complete metric of constant negative curvature in each leaf is continuous in $\mathbb P^2\setminus \mathrm{Sing}(\mathcal{F})$ (the two conclusions hold as soon as the foliation has degree greater than~$2$ and non-degenerate singularities). This allows, given a leafwise path, to homotope it fixing endpoints, to the unique geodesic representative in its class.

 Let $p$ be a point of $U_H=B \setminus \Lambda_H$. The basic technical device is a pair $\gamma= (\gamma^1, \gamma^2)$ of continuous families $\gamma^i = \{ \gamma_t ^i(s) \} _{t\in [0,1]}$ of $s$--leafwise paths satisfying

%This  is equivalent to specifying $F_s$ for the edges  of $T$ corresponding to the generators of the semigroup.
%(In the suspension picture this amounts to specifying the projection to the diagonal,  by first
%specifying it on a fundamental domain,
%then  extending  this equivariantly.)

\begin{itemize}
\item $\gamma_0^i $ is the constant path at $p$,
\item $\gamma_t^i (0) \in U_H$ and $\gamma_1 ^i (0) = h_i (p)$
\item $\gamma_1 ^i$ is the leafwise geodesic path starting at $h_i(p)$ and ending at $p$, and carrying the holonomy map $h_i^{-1}$.
%\item $\gamma_1 ^i$ ENDPT
\end{itemize}
To ensure existence of such families, for $i=1,2$, we add the special condition;
$$\{\gamma_t^i (1) :t\in [0,1]\} \subset \{\gamma_t^i (0) :t\in [0,1]\} \cup \{\gamma_1^i (s) :s\in [0,1]\},$$
i.e.
start with the constant path at $p$, and begin moving continuously to the constant path at $h_i(p)$ in the space of $U_H $-valued constant paths. Then, construct a homotopy (with one free extremity) between the constant path at $h_i(p)$ and the leafwise geodesic path going from $h_i(p)$ to $p$.

%Note also that
%we use $\gamma_t^i (1)$, and better yet, $\gamma_t^i (t)$, to represent a homotopy-theoretic %diagonal, which skirts the edges of a rectangle
%rather than crossing it geometrically--for lack of an intrinsic product structure in the %ambient space.

\begin{figure}

\centering

\includegraphics[width=12cm]{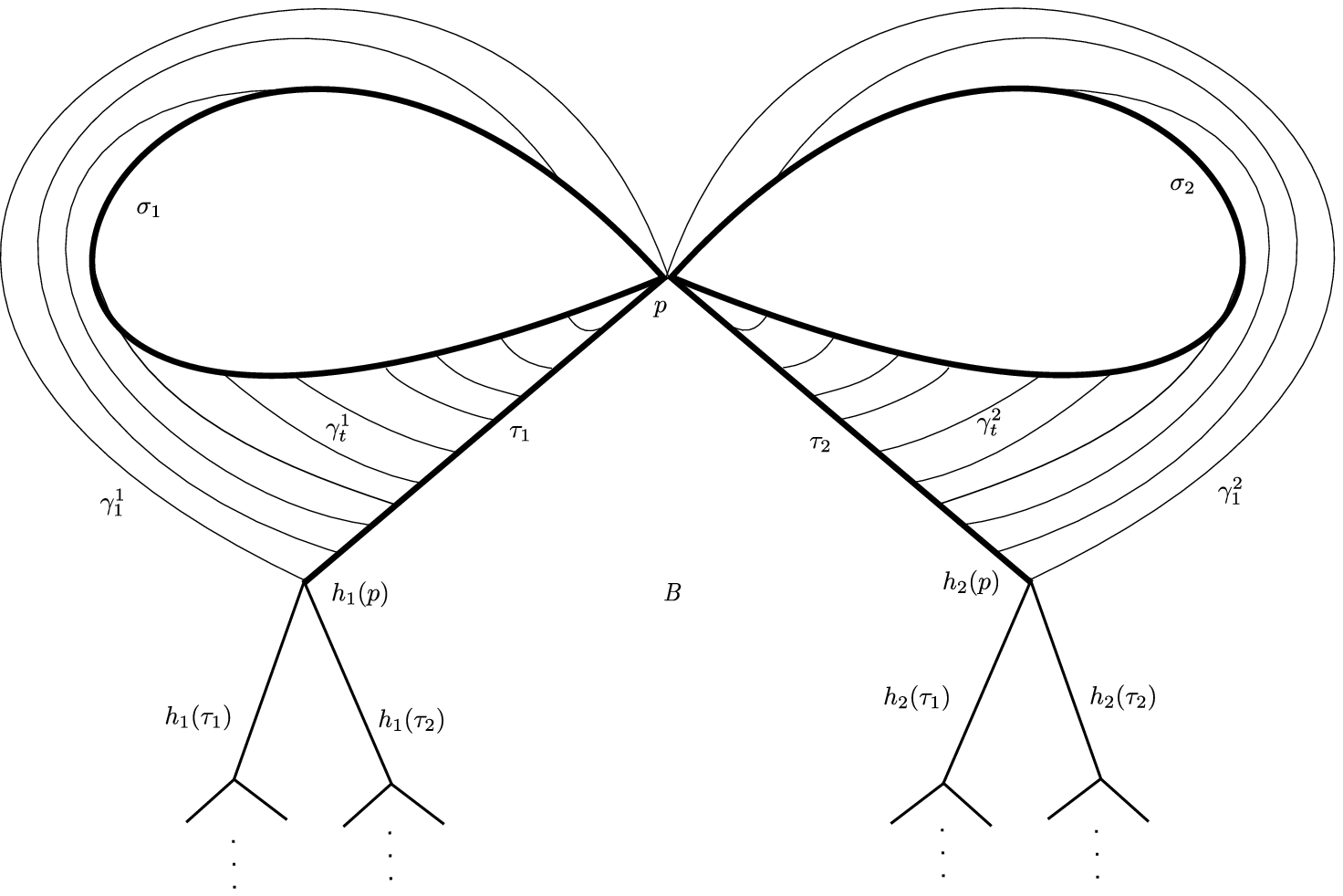}

\caption{The $\mathcal{F}$-homotopy between $\tau_i$ and $\sigma_i$ is realized by the finer paths}\label{fig:papillon}

\end{figure}

Given $\gamma$, we now indicate how to construct associated continuous maps $F_j ^{\gamma} : T \rightarrow \mathbb P^2$, for $j=0,1$. For $i=1,2$, denote by $\tau_i: [0,1] \rightarrow U_H$ the path $\tau_i (t) = \gamma_t ^i (0)$ going from $p$ to $h_i(p)$.
We extend the $\tau_i$'s as a continuous $\rho$-equivariant map $F_0^{\gamma} : T \rightarrow B$. We refer to its image as the tree generated by $\tau_1\vee\tau_2$. Notice that for each point $q\in F_0^{\gamma}(T)$ there exist unique $i\in\{1,2\}$, $t\in[0,1)$ and a finite composition $h_q$ of $h_1$ and $h_2$'s of maximal length such that
\begin{equation}
  \label{eq:paramtree}
h_q(\tau_i(t))=q
\end{equation}

 Parametrize $Q$ by $[0,1] \cup [0,1]$ with all endpoints identified.
Now, define for $i=1,2$, $\sigma_i : [0,1] \rightarrow \mathbb P^2$ as the map $\sigma_i (t) = \gamma_t ^i (t)$. Obviously $\sigma_i (0) = \sigma_i (1) = p$. Thus, we can define a map $\sigma_1 \vee \sigma_2$ from $Q$ to $\mathbb P^2$, which will be referred to as the diagonal. Then we obtain a $\pi_1(Q)^+$-invariant continuous map $F_1^{\gamma} : T \rightarrow \mathbb P^2$ by composing $\sigma_1 \vee \sigma_2$ with the
restriction  to $T$ of the universal covering map of $Q$.

\begin{figure}

\centering

\includegraphics[width=12cm]{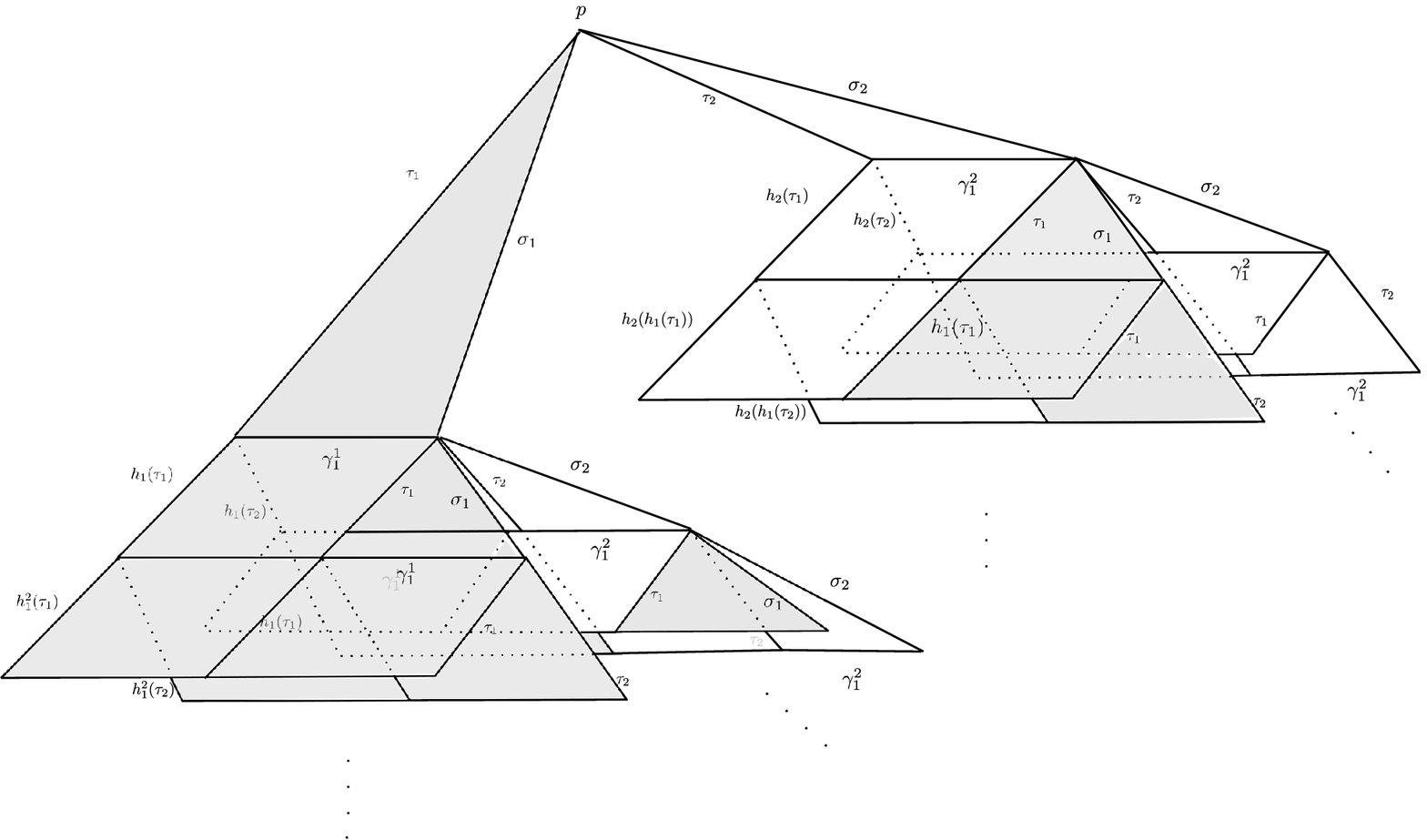}

\caption{The $\mathcal{F}$-homotopy between $F_0^{\gamma}$ and $F_1^{\gamma}$. Shadowed squares correspond to follow $\Gamma^1$, the $h_1^{-1}$-tube. White squares correspond to follow $\Gamma^2$, the $h_2^{-1}$-tube. Shadowed triangles follow $\gamma^1 $, the leafwise homotopy joining $\tau_1$ with the diagonal $\sigma_1$. White triangles follow $\gamma^2$,  the leafwise homotopy joining $\tau_2$ with the diagonal $\sigma_2$.}\label{fig:escalier}

\end{figure}

 The family $\gamma$ provides an $\mathcal F$-homotopy between generators $\tau_1\vee\tau_2$ of the tree and the generators $\sigma _ 1\vee\sigma_2$ of the diagonal (see Figure \ref{fig:papillon}). We will now extend this homotopy and get an $\mathcal F$-homotopy between $F_0 ^{\gamma}$ and $F_1^{\gamma}$. The basic idea is that, by using the relation (\ref{eq:paramtree}), any point $q$ in the tree can be joined by a leafwise path in the tube of $h_q^{-1}$ (a concatenation of the \emph{geodesic} tubes of $h_1^{-1}$ and $h_2^{-1}$) to its corresponding point $\tau_i(t)$. Remark that we follow the tubes in the expanding direction of the holonomy. By this procedure we obtain a family of leafwise paths parametrized by $T$ which is continuous only outside the vertices of $T$. At the vertices of the tree the number of tubes of $h_1^{-1}$ or $h_2^{-1}$ that we follow changes by $1$ because of the conditions defining $h_q$. However, if for each $q$ in the tree instead of considering the leafwise path following the expanding tubes until $\tau_i(t)$ we go a little further by concatenating the path $\gamma_t^i$, the family will also be continuous at the vertices, since $\gamma_1^i$ and the leafwise path in the tube of $h_i^{-1}$ are both the same geodesic. The endpoints of the new family will be in the diagonal instead of being in $B$.

 To avoid cumbersome notations of homotopy parametrizations we have preferred to represent the  homotopy between $F_0^{\gamma}$ and $F_1^{\gamma}$ by a figure. For each $i=1,2$ consider the holonomy tube $\Gamma^i:[0,1] \times h_i(B)\rightarrow \mathbb P^2$ whose associated holonomy map is $h_i^{-1}$, given by Lemma~\ref{lem:schottky}. To get continuity it is important here to use the leafwise geodesics for the $\Gamma^i (.,q)$'s. We find those tubes  and the family $\gamma$ in Figure (\ref{fig:escalier}) concatenated in the adequate way to produce the desired $\mathcal F$-homotopy between $F_0^{\gamma}$ and $F_1^{\gamma}$.
 Formally, the existence proof uses a straightforward induction on the heights of edges (i.e. lengths of multiindices) in the tree.

 Some remarks are in order at this point:

 First, the choice of the pair $\gamma=(\gamma^1,\gamma^2)$ could have been done by using other more differential--topological conditions instead of using the metric arguments that apply only under the hypothesis that all leaves are of hyperbolic type. These other choices allow to guarantee the continuity of the $\mathcal{F}$-homotopies to more general foliations.

 Second, let us relate the diagonal construction of section \ref{s:riccati} in the Riccati context to the present context. In the former case the diagonal can be interpreted in the covering space, which is a product, as the graph of a holomorphic map from a leaf to a transversal. The asymptotical values of this map provide singular points for the holonomy germ determined by a germ of its inverse. In other words there are infinite paths in the diagonal whose projection to the fibre converge, while the projection onto the leaf goes towards infinity. In the present case the diagonal is represented by the map $F_1^{\gamma}:T\rightarrow \mathbb{P}^2$. The basic point is that the diagonal can be projected to the tree $F_0^{\gamma}(T)$ in the transversal $B$ and \emph{also} to a tree in the leaf of $\mathcal{F}$ through $p$, by using the holonomy tubes of $h_1$ and $h_2$ to define lifts. If we take an infinite path in $T$ towards a point in the boundary, its projection from the diagonal to $B$ converges, while its projection onto the leaf diverges.

The proof is complete except that $F_1^{\gamma}$ constructed above
does not necessarily take values in an algebraic curve of $\mathbb P^2$, so we now modify the tubes $\gamma$ to obtain this latter condition. For this purpose, we will construct in what follows a perturbation $\hat{\sigma_1} \vee \hat{\sigma_2}$ of $\sigma_1 \vee \sigma_2$ in the uniform topology, based at the point $p$, so that its image is contained in an algebraic curve. Then
the construction of a pair $\hat{\gamma}=(\hat{\gamma}_1, \hat{\gamma}_2)$ whose corresponding diagonal is $\hat{\sigma}_1 \vee \hat{\sigma}_2$ follows by a simple perturbation of the construction above.
One should note that the  perturbation $\hat{\sigma}_1 \vee \hat{\sigma}_2$
determines uniquely  a perturbation of $\tau_i$
compatible with the  $\gamma^i$ construction, which then  makes the continuous dependence of
the  $\gamma^i $ quite obvious.

\begin{lemma}\label{lem:approx}

There exists a continuous mapping $\hat{\sigma}_1 \vee \hat{\sigma}_2:Q\rightarrow \mathbb P^2$ arbitrarily close to $\sigma_1 \vee \sigma_2$ in the uniform topology, whose image is contained in a smooth algebraic curve of $\mathbb P^2$, and both maps share the same basepoint $p\in\mathbb P^2$.

\end{lemma}

\begin{proof} Let $\varepsilon >0$. For $i=1,2$ choose $\tilde{\sigma} _i:S^1 \rightarrow {\mathbb C} ^2 \subset \mathbb P^2$ a $C^{\infty}$ map in a $\varepsilon /4$-neighbourhood of $\sigma_i$ satisfying $\sigma_i(1)=\tilde{\sigma}_i(1)=p$. We can assume that $\tilde{\sigma}'_1(1),\tilde{\sigma}'_2(1)$ are $\mathbb C$-linearly independent. The Fourier series of $\tilde{\sigma} _i$ (resp. the series' first derivative) converges uniformly to $\tilde{\sigma}_i$ (resp. $\tilde{\sigma}_i'$) so by truncating the series for some sufficiently big $N>0$ and applying a small translation we get a finite family $\{a_n^i,b_n^i\in\mathbb C :|n|<N\}\subset\mathbb C$ and a map

$$\bar{\sigma}_i(e^{2i\pi t})=\left(\sum_{|n|<N}a_n^i (e^{2i\pi t})^n,\sum_{|n|<N}b_n^i (e^{2i\pi t})^n \right),$$
which is in a $\varepsilon/2$-neighbourhood of $\sigma_i$ and satisfies $\bar{\sigma}_i(1)=p$. The map $\bar{\sigma}_i$ is the restriction of a rational map $\mathbb P^1\rightarrow \mathbb P^2$ to $S^1\subset \mathbb P^1$ so its image is contained in an algebraic curve $C_i=\{f_i=0\}\subset \mathbb P^2$ that contains $p$. The non-collinearity of $\tilde{\sigma}'_1(1)$ and $\tilde{\sigma}'_2(1)$ guarantees that $p$ is a normal crossing point of the reducible curve $C_1\cup C_2=\{f_1 f_2=0\}$. For a generic and sufficiently small $\delta\in \mathbb C$, the curve $\tilde{C}=\{f_1 f_2=\delta\}$ is smooth and $\varepsilon/2$-close to $C_1\cup C_2$. In a neighbourhood of $p$, $\tilde{C}$ is arc connected (it is biholomorphic to a disc with a finite number of holes) so we can find a point $q\in\tilde{C}$ which is $\varepsilon/2$-close to $p$ and two loops $\breve{\sigma}_1, \breve{\sigma}_2$ in $\tilde{C}$ with basepoint at $q$ which are $\varepsilon/2$-close to $\sigma_1$ and $\sigma_2$ respectively. By choosing an automorphism $\varphi$ of $\mathbb P^2$ which is $\varepsilon/4$-close to the identity sending $q$ to $p$ we have that the curve $C:=\varphi(\tilde{C})$ and the loops $\hat{\sigma}_i:=\varphi\circ\breve{\sigma}_i$ satisfy the desired properties.
\end{proof}

\subsection{Proof of Theorem~\ref{thm:curveofsing}}

A careful analysis of the proof of Theorem \ref{thm:cantor} shows that if $B$ is a transverse disc to a holomorphic singular foliation $\mathcal{F}$ of $\mathbb P^2$ having only hyperbolic leaves and $h_1,\ldots,h_l:B\rightarrow B$ are holonomy maps such that
\begin{itemize}
  \item there exists a constant $c<1$ satisfying $|h_i'(x)|<c$ for all $x\in B$ and $i\in\{1,\ldots,l\}$, and
  \item each $h_i$ carries a holonomy tube
\end{itemize}
then we can find an algebraic curve $C$ and a germ of holonomy of $\mathcal{F}$ from $B$ to $C$ for which every point in the limit set of the action of the semigroup $H$ generated by $h_1,\ldots,h_l$ on $B$, namely
$$\Lambda_H=\bigcup_{(i_n)\in\{1,\ldots,l\}^{\mathbb{N}}}\left( \bigcap_{n\geq 0}h_{i_1}\circ\cdots\circ h_{i_n}(B)\right),$$ is a singularity for its analytic continuation. In fact we only need to remark that the leafwise Poincar\'e metric is continuous (\cite{Lins}), adapt Proposition~\ref{p:F-homotopy} by taking $T$ as an $l$-adic tree instead of a dyadic tree, $Q$ as a bouquet of $l$ circles, and use an adaptation of Lemma \ref{lem:approx} to guarantee that arbitrary close to a continuous map $Q\rightarrow\mathbb P^2$ there is another such map with image contained in an algebraic curve.

The plane algebraic foliations that will satisfy Theorem~\ref{thm:curveofsing} are examples that we learned from Loray and Rebelo's paper~\cite{LR}. To find the admissible germ for such a foliation, we will construct a semi-group in the holonomy pseudo-group, defined by some holomorphic maps carried by holonomy tubes, whose limit set has non empty interior, as above. For this purpose, we review Loray/Rebelo's techniques.

The important property satisfied by these foliations is that their pseudo-group is \textit{non discrete}: namely, there exist transverse discs $B\subset B'$, with $B$ relatively compact in $B'$, and a sequence of holonomy maps $g_n : B \rightarrow B'$ which are different from the identity for every integer $n$, but nevertheless converge to the identity uniformly in $B$. The main ingredient to ensure that a foliation has a non discrete pseudo-group is~\cite[Lemma 3.3]{LR}:

\begin{lemma} \label{l:LorayRebelo} There is a constant $\varepsilon_0>0$ such that the following holds. Let $f: \mathbb D\rightarrow \mathbb D$ be a map defined by $f(z) = \lambda z$ for some complex number $\lambda$ of modulus $|\lambda | < 1$ and such that $|\lambda - 1| \leq \varepsilon_0$. Let $g: \mathbb D \rightarrow \mathbb C$ be an injective holomorphic map, $\varepsilon_0$-close to the identity, i.e. $\sup _{z\in \mathbb D} |g(z)-z| \leq \varepsilon_0$. Then, there exists an integer $N\geq 1$ such that the holomorphic maps defined by $g_0= g$, and $g_{k+1} = f^{-N} \circ [f,g_k]\circ f^{N}$ for every integer $k\in \mathbb N$ have domain of definition at least the disc $\mathbb D_{1/3}$, and converge uniformly to the identity when $k$ tends to infinity. Moreover, if $g$ is not an affine map, then $g_k$ is different from the identity for every $k$.
 \end{lemma}

In what follows, we introduce the set $\mathcal U_d$ of degree $d$ plane algebraic foliations, whose pseudo-group contains two maps $f$ and $g$ defined on some transversal set $B\simeq \mathbb D \subset B'\simeq \mathbb C$ as in Lemma~\ref{l:LorayRebelo}, and such that moreover:
\begin{itemize}
\item $\lambda$ is not real,
\item $g(0)$ is different from $0$,
\item $g$ is not affine in the linearization coordinate of $f$,
\item there exist holonomy tubes carrying the maps $f$ and $g$,
\item the singularities are of hyperbolic type, and the foliation do not carry any invariant algebraic curve. In particular, the leaves are hyperbolic, and the Poincar\'e metric on the leaves is continuous.
\end{itemize}

All the above conditions are open (there is a little work to see that after perturbation $g$ is still not affine in the linearization coordinate of $f$; this is due to the fact that this coordinate depends continuously on $f$); hence $\mathcal U_d$ is an open set. Loray/Rebelo~\cite{LR} proved that this set is not empty. The idea is to construct a particular foliation satisfying all the conditions but the last one, and to perturb it so that the singularities become hyperbolic, and no algebraic curve is invariant. Such foliations are for instance the Riccati foliations with dense monodromy group: the reader can check that they satisfy all the conditions, appart the last one.

At this point, an important thing to observe is that, if a foliation belong to $\mathcal U_d$, the corresponding maps $g_k$ defined on $\mathbb D_{1/3}$ are carried by holonomy tubes for every $k$. These tubes are obtained by concatenating the tubes of the maps $f$, $g$, $f^{-1}$, $g^{-1}$.

Next, we review another fundamental result proved in Loray/Rebelo's paper: the existence of flows in the closure of the pseudo-group of a foliation in $\mathcal U_d$. Here is the precise statement (see~\cite[Corollary 4.2]{LR}):

\begin{lemma} \label{l:flows}
Let $\mathcal F$ be a foliation belonging to $\mathcal U_d$ and $\mathcal G$ be the holonomy pseudo-group of $\mathcal F$. Then there is a transversal $B$ isomorphic to the unit disc, and two holomorphic vector fields $X_1, X_2$ defined on $\mathbb D$, such that $X_1(0)$ and $X_2(0)$ are independent over the real numbers, and such that for every small enough real number $t$, the map $\exp (t X_i)$ is defined on the disc $\mathbb D_ {1/2}$ (with values in $\mathbb D$) and is a uniform limit of a sequence $k_n$ of elements of $\mathcal G$ on $\mathbb D_{1/2}$ (in particular the domain of definition of $k_n$ contains $\mathbb D_{1/2}$ for every $n$). Moreover, the maps $k_n$ are associated to holonomy tubes.
\end{lemma}

Remark that Loray/Rebelo's Corollary 4.2 states that there is at least one vector field $X_1$ satisfying Lemma~\ref{l:flows} such that $X_1(0) \neq 0$. Then, by taking $X_2 = f_* X_1$, we get another vector field, and these two verify lemma~\ref{l:flows}. Notice also that the maps $k_n$'s are carried by holonomy tubes. Indeed, they are constructed by considering compositions of the form $f^{- N_k} \circ g_k \circ f^{N_k}$, for some integer $N_k$, where $g_k$ are the maps constructed in Lemma~\ref{l:LorayRebelo}. These maps are as above carried by holonomy tubes: the concatenations of the holonomy tubes of the maps $f$, $f^{-1}$ and $g_k$'s.

 The following lemma (together with the sketched construction above) finishes the proof:

\begin{lemma}\label{lem:dense semi-group}
Let $\mathcal F$ be a foliation as in Lemma~\ref{l:flows}. Then there exists a finite number of holonomy maps  $h_1, \ldots, h_l$ from a transversal disc $B \simeq \mathbb D$ to itself, carried by holonomy tubes, such that the limit set of the action of the semi-group $H$ generated by the $h_i$'s on $\mathbb D$ has nonempty interior. Moreover the $h_i$'s can be chosen to be contracting, i.e. there is a constant $0<c<1$ such that for every $i= 1,\ldots, l$ and $x\in \mathbb D$, $| h_i' (x) | <c$.
\end{lemma}

\begin{proof} Let $h_1, \ldots h_l : \mathbb D \rightarrow \mathbb D$ be some maps and $H$ the semi-group generated by these maps. Suppose that there is a point $p$ in $\mathbb D$, and an infinite sequence of indices $i_1, i_2, ..., i_n, ...$ each of them belonging to $\{1,...,l\}$ such that we can define the sequence of points $p_1= h_{i_1}^{-1}(p)$ and for every $n>1$, $p_n = h_{i_n}^{-1} (p_{n-1})$. If we suppose further that for every $i= 1,\ldots, l$ and $x\in \mathbb D$, $| h_i' (x) | <c<1$ then actually $p=\cap_{n\geq 1}h_{i_1}\circ\cdots\circ h_{i_n}(\mathbb D)$. By definition $p$ is a limit point of the semigroup. Our goal will be to find such maps $h_1, ... , h_l$ in the holonomy pseudo-group $\mathcal G$ of $\mathcal F$, defined on $\mathbb{D}$, and such that for every point in a nonempty open set~$U$ of~$\mathbb{D}$ we can find a sequence $i_1,... , i_n, ..$ as above.

Observe that because $X_1$ and $X_2$ are not collinear at $0$, there is a neighborhood of~$0$, say $\mathbb D_{r}$ with $r>0$, and a positive real number $ \tau >0$, such that for every point $p$ of $\overline{\mathbb D _{r}}$, there are unique real numbers $t_1, t_2$ such that $|t_1|, \ |t_2 | < \tau$ and $0 = \exp ( t_1 X_1) \exp (t_2 X_2) (p)$. Moreover, we can suppose that the composition $\exp( t_1 X_1 ) \exp(t_2 X_2)$, and its inverse $\exp( -t_2 X_2 ) \exp (-t_1 X_1)$, are defined on $\mathbb D _{1/4}$ with values in $\mathbb D$, for every $|t_1|,|t_2|\ < \tau$.

We fix some constants. Denote $M>0$ an upper bound for all the derivatives
\[ |\big( \exp (-t_2 X_2) \exp (-t_1 X_1) \big) '(p)|\]
for every point $p$ in the ball of radius $1/4$ and every $|t_1|,|t_2|\leq \tau$. Let $k$ be an integer such that $ |\lambda |^k < \min (r, 1/8, 1/2M)$.

Given $p\in \overline{\mathbb D_{r}}$, we approximate the maps $\exp(t_j X_j)$, $j=1,2$ on $\mathbb D_{1/2}$ by elements $k_j$ of the pseudo-group defined on $\mathbb D_{1/2}$, so that the map $g= k_1 k_2$ is defined on $\mathbb D_{1/8}$ with derivative bounded by $2M$, and $g(p)$ belongs to the disc $\mathbb D _{|\lambda|^k r}$. If $k_j$ is sufficiently close to $\exp(t_j X_j)$, we will also get that $g^{-1}$ is defined on $\mathbb D_{1/8}$, and that its derivative is bounded by $2M$.

By compactness of $\overline{\mathbb D _r}$, we can find a finite number of elements $g_1,\ldots, g_l$ of $\mathcal G$ defined on $\mathbb D_{1/8}$, such that for every point $p$ in $\mathbb D _ {r}$, there is an index $i\in\{1,\ldots ,l\}$ such that $g_i (p) $ belongs to $\mathbb D _{|\lambda|^k r}$. Moreover, the map $g_i$'s, together with their inverses, are defined on $\mathbb D_{1/8}$ with derivatives bounded by $2M$.

Consider the maps $h_i = g_i ^{-1} \circ f^k$, for $i= 1,\ldots , l$. Because $|\lambda |^k < 1/8$, $h_i$ is a map from $\mathbb D $ to itself. Moreover, the derivative of $h_i$ is uniformly bounded by $2M |\lambda|^k <1$. We have $h_i^{-1} = f^{-k} \circ g_i$, i.e. $h_i^{-1} = \lambda^{-k} g_i^{-1}$. Since for every $p\in U = \mathbb D_{r}$ there is an index $i\in\{1,\ldots ,l\}$ such that $g_i (p) $ belongs to $\mathbb D_{|\lambda|^k r}$ we have that $h_i^{-1} (p) $ belongs to $U$. We define $i_1:=i$ and repeat the argument for $h_{i_1}^{-1} (p)$. By induction we to get our sequence $i_1,\ldots, i_n, \ldots$. The proof is complete. \end{proof}

%%%%%%%%%%%%%%%%%%%%%%%%%%%%%
\begin{small}

\vspace{0.3cm}

G. Calsamiglia,
Instituto de Matem\' atica, Universidade Federal Fluminense, Rua M\'ario Santos Braga s/n, 24020-140, Niter\'oi, Brazil,
gabriel@mat.uff.br

\vspace{0.3cm}

B. Deroin,
D\'epartement de Math\'ematiques d'Orsay, Universit\'e Paris 11, 91405 Orsay Cedex, bertrand.deroin@math.u-psud.fr

\vspace{0.3cm}

S. Frankel,
sid1frankel$<at>$gmail.com,   (and frankel$<at>$math.jussieu.fr, spring 2010 )

\vspace{0.3cm}

A. Guillot,
Instituto de Matem\'{a}ticas, Unidad Cuernavaca, Universidad Nacional Aut\'{o}noma de M\'{e}xico, A.P.~273-3 Admon.~3, Cuernavaca, Morelos, 62251 Mexico,
adolfo@matcuer.unam.mx

\end{small}


\begin{thebibliography}{99}

\bibitem {ahl} {\sc L. V. Ahlfors.} Lectures on Quasiconformal Mappings. 2nd Edition, {\em AMS Bookstore}, 2006.

\bibitem {Barth-al} {\sc W. Barth \& K. Hulek \& C. Peters \& A. Van de Ven.} Compact complex surfaces.
Second edition. {\em Springer-Verlag}, Berlin, 2004.

\bibitem {Beardon} {\sc A. F. Beardon.} The geometry of discrete groups. Grad. Texts in Math. 91, \emph{Springer}, New York, 1983.

\bibitem {BLL} {\sc M. Belliart, I. Liousse \& F. Loray.} The generic rational differential equation $dw/dz=P_n(z,w)/Q_n(z,w)$ on $\mathbb P^2$ carries no interesting transverse structure. {\em Ergodic Theory Dynam. Systems} {\bf  21}  (2001),  no. 6, 1599--1607.

\bibitem {BLM} {\sc C. Bonatti, R. Langevin and R. Moussu.} Feuilletages de ${\mathbb  P}^n$~: de l'holonomie
              hyperbolique pour les minimaux exceptionnels, {\em Inst. Hautes \'Etudes Sci. Publ.Math.}, nr. 75, 1992, p. 123-134.

\bibitem {Br} {\sc M. Brunella.} Birational theory of foliations, Monograf\'{\i}as de Matem\'{a}tica. Instituto de Matem\'{a}tica Pura e Aplicada (IMPA), Rio de Janeiro, 2000.

%\bibitem {cam}
%{\sc C. Camacho \& P. Sad.} Invariant varieties through holomorphic vector fields, {\em Ann. Math.} 115 (1982), p. 579-595.

\bibitem {CLS}{\sc C. Camacho, A. Lins Neto \& P. Sad.}   Minimal sets of foliations on complex projective spaces. {\em Inst. Hautes \'Etudes Sci. Publ. Math.} No. 68 (1988), 187-203.

%\bibitem {CS} {\sc M. Culler \& P. B. Shalen.} Varieties of group representations and splitting of $3$-manifolds. {\em Ann. of Math.} {\bf 117} (1983) no. 1, 109-146.

\bibitem {D-O} {\sc K. Diederich  \& T. Ohsawa.} Harmonic mappings and disc bundles over compact Kähler manifolds. {\em Publ. RIMS Kyoto Univ. } {\bf 21 } (1985) 819–833.

\bibitem {GKM} {\sc D. Gallo, M.  Kapovich \& A. Marden.}
The monodromy groups of Schwarzian equations on closed Riemann surfaces. {\em Ann. of Math.} (2) {\bf 151} (2000), no. 2, 625-704.

%\bibitem {Guillot} {\sc A. Guillot.} Semicompleteness of homogeneous quadratic vector fields. \emph{Ann. Inst. Fourier (Grenoble)} \textbf{56} (2006), no. 5, 1583--1615.

\bibitem{HP} {\sc M. Heusener \& J. Porti.} The variety of characters in $\mathrm{PSL}(2,{\mathbb C})$. {\em Bol. Soc. Mat. Mexicana} (3) 10 (2004), p. 231-237.

\bibitem {Hille} {\sc E. Hille.} Ordinary differential equations in the complex domain. Pure and Applied Mathematics.
{\em Wiley-Interscience}, New York-London-Sydney, (1976).

\bibitem{Il1}{\sc Yu. Ilyashenko.} Centennial history of Hilbert's 16th problem. {\em Bull. Amer. Math. Soc.} (N.S.) 39 (2002), no. 3, p.301-354

\bibitem{Il} {\sc Yu. Ilyashenko.}  Some open problems in real and complex dynamical systems. {\em Nonlinearity} 21 (2008) no.7, p. 101-107.

%\bibitem {ince} {\sc E.~L. Ince.} Ordinary differential equations.  {\em Dover Publications}, New York, 1944.

\bibitem {Jouanolou} {\sc J.-P. Jouanolou.} \'Equations de Pfaff alg\'ebriques, {\em Lecture Notes in Math.}, 708, Springer, Berlin, 1979.

\bibitem{Lins} {\sc A. Lins-Neto.}  Uniformization and the Poincar\' e metric on the leaves of a foliation by curves, {\em Bol. Soc. Brasil. Mat. (N.S.)} 31, no.3 (2000), p. 351-366.

\bibitem{LinsNeto-Pereira} {\sc J. Vitorio Pereira \& A. Lins Neto.}  On the Generic Rank of the Baum-Bott Map. {\em Compositio Math.} Vol 142, no. 6, p.  1549-1586, 2006.

%\bibitem{LS} {\sc A. Lins Neto \& M. G. Soares} Algebraic solutions of one-dimensional foliations. {\em J. Differential Geom.} 43 (1996), no. 3, 652–673.

\bibitem {Loray} {\sc F. Loray.} Sur les th\'eor\`emes I et II de Painlev\'e. {\em Contemporary Mathematics} 389 (2005) p. 165-190.

%\bibitem {Loray-Marin} {\sc F. Loray and D. Mar{\'{\i}}n.} Projective structures and projective bundles over compact Riemann surfaces. Ast\'erisque \textbf{323} (2009) p.~223-252.

\bibitem{LR} {\sc F. Loray \& J. Rebelo.} Minimal, rigid foliations by curves on $\mathbb P^n$. {\em J. Eur. Math. Soc.}, {\bf 5} (2003) p.~147–201.

\bibitem {Marin} {\sc D. Mar{\'i}n.} Problemas de M\'{o}dulos para una clase de foliaciones holomorfas, Tesis de Doctorado, Universitat Aut\`{o}noma de Barcelona, (2002) {\em  http://www.tdcat.cbuc.es/TDX-0528101-120615/}, p.113-146.

\bibitem {Painleve} {\sc P. Painlev\'e.} \OE uvres compl\`etes. \'Editions du CNRS, Paris, 1972.


\bibitem {Palm} {\sc C.F.B. Palmeira.} Open manifolds foliated by planes. {\em Annals of  Mathematics} 107 (1978), 109-131.


\bibitem {Zak} {\sc S. Zakeri.} Dynamics of singular holomorphic foliations on the
complex projective plane, in  Laminations and foliations in dynamics,
geometry and topology (Stony Brook, NY, 1998), {\em Contemporary
Mathematics}, 269, Providence 2001, p. 179--233.

\bibitem {Zas} {\sc H. Zassenhaus.} Beweis eines Satze \"uber diskrete Gruppen. Abhandlungen aus dem Mathematischen Seminar der Universit\"at Hamburg. Vol 19 (1938), p. 289-312.

\end{thebibliography}
\end{document}